\documentclass[11pt]{amsart}
\usepackage{ifthen}
\usepackage{verbatim, latexsym, amssymb, amsmath}
\usepackage{enumerate}
\usepackage{etoolbox}
\usepackage[colorlinks=true]{hyperref}
\numberwithin{equation}{section}
\AtBeginEnvironment{proof}{\vspace{-14pt}}

\def\ve{\varepsilon}
\def\phi{\varphi}
\def\R{\mathbb R}
\def\N{\mathbb N}
\def\R{\mathbb R}
\def\N{\mathbb N}

\def\H{\mathcal H}

\def\dist{\mathrm{dist}}

\def\d{\mathrm{div}}
\def\dist{\mathrm{dist}}

\def\l{\mathrm{loc}}

\newtheorem{thm}{Theorem}[section]
\newtheorem{lemm}[thm]{Lemma}
\newtheorem{cor}[thm]{Corollary}
\newtheorem{prop}[thm]{Proposition}
\theoremstyle{remark}
\newtheorem{rmk}[thm]{Remark}
\newtheorem{rmks}[thm]{Remarks}

\theoremstyle{definition}
\newtheorem{defi}[thm]{Definition}

\def\uclhome{@ucl.ac.uk}
\def\ethhome{@math.ethz.ch}
\def\imphome{@imperial.ac.uk}

\begin{document}

\title{On short time existence for the Planar Network Flow}

\author{Tom Ilmanen}
\thanks{}
\address{Tom Ilmanen: Departement Mathematik, R\"amistrasse 101, 8092
  Z\"urich, Switzerland}
\email{tom.ilmanen\ethhome}

\author{Andr\'e Neves}
\thanks{}
\address{Andr\'e Neves: University of Chicago, Department of Mathematics, Chicago IL 60637, USA /Imperial College London, Huxley Building, 180 Queen's Gate, London SW7 2RH, United Kingdom}
\email{aneves@uchicago.edu, aneves\imphome}

\author{Felix Schulze}
\thanks{}
\address{Felix Schulze: 
  Department of Mathematics, University College London, 25 Gordon St,
  London WC1E 6BT, United Kingdom}
\email{f.schulze\uclhome}




\begin{abstract}
We prove the existence of the flow by curvature of regular planar
networks starting from an initial network which is non-regular. The
proof relies on a monotonicity formula for expanding solutions
and a local regularity result for the network flow 
in the spirit of B. White's local regularity theorem for mean
curvature flow. We also show a pseudolocality theorem for mean
curvature flow in any codimension, assuming only that the initial
submanifold can be locally written as a graph with sufficiently small
Lipschitz constant.
\end{abstract}

\pagestyle{headings}

\maketitle 

\markboth{On short time existence for the Planar Network Flow} {Tom Ilmanen, Andr\'e Neves, and Felix Schulze}

\section{Introduction}
\thispagestyle{empty}
A natural generalization of the flow of smooth hypersurfaces by
 mean curvature is the flow of surface
clusters, where three hypersurfaces can meet under equal angles,
forming a liquid edge. These edges then again can meet on lower
dimensional strata. The simplest such configuration, which already includes many
aspects of the situation in higher dimensions, is the flow by
curvature of a network of curves in the plane. 

In brief, we consider a planar network to be a finite union of
embedded curves of non-zero length, which only intersect at their endpoints. 
We require that at each such point, called a multiple point, a finite
number, but at least two endpoints come together. We call a network {\it regular}
if at each multiple point three ends of segments meet, forming angles of
$2\pi/3$. Without this condition, but requiring that the segments have
mutually distinct exterior unit tangents at each multiple point, we
call such a network {\it non-regular}. A solution to the planar
network flow is a smooth family of regular, planar networks, such that
the normal component of the speed under
the evolution at every point on each segment is given by the curvature
vector of the segment at the point. For a more precise definition see
section \ref{sect:defi}.

Since the evolution by curvature of a regular network is the gradient flow of
the length functional it is natural to assume that at regular times
only triple points are present and the angles formed by the segments are balanced. This is
supported by
the fact that only the balanced configuration with three segments
meeting infinitesimally minimizes length around each multiple point,
if one allows as competitors connected networks with additional segments. 

After the pioneering work of Brakke \cite{Brakke}, whose definition
of moving varifolds includes the evolution of networks described here, and the
fundamental results on curve shortening flow of embedded closed curves by
Gage/Hamilton and Grayson \cite{GageHamilton86, Grayson89}, the first
thorough analytical treatment of the flow of networks was undertaken by
Mantegazza, Novaga and Tortorelli \cite{MNTNetworks}. Aside from
establishing  short time existence of the network flow starting
from a regular initial network, their focus is mainly on the evolution of
three arcs with three fixed endpoints, meeting at one interior triple
point. In this special setting they obtain long time existence and convergence
under certain hypotheses. In a recent preprint by Magni, Mantegazza and
Novaga \cite{MagniMantegazzaNovaga13} it is shown that these
hypotheses are actually fulfilled, provided none of the arcs contracts to
zero length. Existence and convergence
properties of the network flow in other
special configurations have been studied in
\cite{BellettiniNovaga11, HaettenschweilerDiplom07, MazzeoSaez07,
  SchnuererSchulze_2007, Linsen}. 

It is conjectured that at a singular time of the flow no tangent flow
which is a static line of higher multiplicity can develop. An immediate
consequence of this conjecture is that at any singular
time, the length of one of the segments shrinks to zero, and at least two
triple junctions collide. It thus can be expected that at the singular
time a non-regular network forms.  

In the present paper we show that starting from a non-regular initial
network, a smooth evolution of
regular networks exists. In this evolution it might happen that 
out of non-regular initial multiple points new segments are created.

\begin{thm}
  \label{thm:short time ex}
Let $\gamma_0$ be a non-regular, connected planar network with bound\-ed curvature
as in Definition \ref{def:network}. Then there exists $T>0$ and a
smooth connected solution of the planar network flow of regular networks
$(\gamma_t)_{0<t<T}$ such that $\gamma_t \rightarrow \gamma_0$ in the
varifold topology as $t \searrow 0$. Away from the non-regular
multiple points of $\gamma_0$ the convergence is in
$C^\infty_\text{loc}$. Furthermore, there exists a constant $C > 0$
such that
$$\sup_{\gamma_t}|k| \leq \frac{C}{\sqrt{t}}$$
and the length of the shortest segment of $\gamma_t$ is bounded from
below by $C^{-1} \sqrt{t}$.
\end{thm}

\begin{rmks}i) The proof uses a local monotonicity formula, which only
  works if the network is locally tree-like, i.e. contains locally no
  loops. In the proof we glue small tree-like self-similarly
  expanding networks into $\gamma_0$ around non-regular multiple
  points. These tree-like, connected, self-similarly expanding networks always exist. Nevertheless, there are also non
  tree-like self-similarly expanding networks, which would correspond to the
  creation of new bounded regions in the complement of the network out
  of non-regular multiple points. Our proof would not work gluing in
  this type of expanders.\\[0.5ex]
ii) The proof of this result presented here yields, only with
  minor modifications, also the corresponding statement for non-regular
  initial networks with bounded curvature and fixed endpoints.\\[0.5ex]
iii) Using a relaxation scheme via the Allen-Cahn equation S\'aez \cite{Saez11}
has shown that regular, smooth solutions starting from a non-regular initial
network, which are tree-like,  and satisfy the estimates on the curvature and
the length of the shortest segment as above, are unique in their
topological class. This yields a corresponding uniqueness statement for our
constructed solutions in this case.
\end{rmks}

The method of proof relies on a monotone integral quantity, which
implies that self-similarly expanding solutions are attractive under
the flow. This monotone integral quantity has already been applied by
the second author in the setting of Lagrangian Mean Curvature Flow in
several places, see \cite{Neves07, Neves13b, Neves13a}. The second
main ingredient is a local regularity result in the spirit of White's
local regularity theorem for smooth mean curvature flow \cite{White05}. 

In the statement of the following local regularity result, $\Theta_{S^1}$ is the Gaussian density of the self-similarly
shrinking circle. Note that $\Theta_{S^1}= \sqrt{(2\pi/e)}>3/2$. The
quantity $\Theta(x,t,r)$ is the Gaussian density at scale $r>0$,
centered at the point $(x,t)$. For details and the definition of
proper flows, see section \ref{sec:reg res}.

\begin{thm}\label{thm:locreg.2}
Let $(\gamma_t)_{t\in[0,T)}$ be a
  smooth, proper and regular planar network flow in $ B_\rho(x_0)\times (t_0-\rho^2,t_0)$ which reaches the point $x_0$ at time
  $t_0 \in (0,T]$. Let $0<\varepsilon, \eta<1$. There exist $C =
  C(\varepsilon, \eta)$ such that if
\begin{equation}\label{eq:locreg.0.5}
\Theta(x,t,r) \leq \Theta_{S^1} -\varepsilon
\end{equation}
for all $(x,t) \in B_\rho(x_0)\times (t_0-\rho^2,t_0)$ and
$0<r<\eta\rho$ for some $\eta >0$, where $(1+\eta)\rho^2\leq
t_0<T$,  then
$$ |k|^2(x,t) \leq \frac{C}{\sigma^2\rho^2} $$
for $ (x,t)\in \big(\gamma_t\cap
    B_{(1-\sigma)\rho}(x_0)\big)\times (t_0-(1-\sigma)^2\rho^2,t_0)$
    and all $\sigma \in (0,1)$.
\end{thm}

\begin{rmks}\label{rem:locreg.2}i)
  One can furthermore show that there is a constant
  $\kappa=\kappa(\varepsilon, \eta)>0$ such that
  the length of each segment which intersects
  $B_{(1-\sigma)\rho}(x_0)\times (t_0-(1-\sigma)^2\rho^2,t_0)$ is
  bounded from below by $\kappa\cdot \sigma \rho$.
 This implies corresponding scaling invariant estimates on all higher
 derivatives of the curvature.\\[0.5ex]
ii) We also prove a corresponding result if the evolving network is
locally tree-like, i.e. does not contain any closed loops of length
less than $\delta >0$ and the Gaussian density ratios are bounded from
above by $2-\varepsilon$.\\[0.5ex]
iii) Note that the result implies that any regular smooth flow, which
is sufficiently close in measure to the static configuration consisting
of three half-lines meeting under equal angles, is smoothly
close. Recently Tonegawa and Wickramasekera
\cite{TonegawaWickramasekera14} have shown that this is also true for integer Brakke flows.
\end{rmks}

To get sufficiently good local control away from the multiple junctions
we also show the following pseudolocality theorem. Since it also holds
for mean curvature flow, we formulate it in full generality.  A similar estimate
assuming initial control
on the second fundamental form has been shown by Chen and Yin
\cite{ChenYin07} and assuming control on up to fourth derivatives by
Brendle and Huisken \cite{BrendleHuisken13}.

In the following, for any point $x \in \R^{n+k}$ we write $x = (\hat{x},
\tilde{x})$ where $\hat{x}$ is the orthogonal projection of $x$ on the
$\R^n$-factor and $\tilde{x}$ the orthogonal projection on the $\R^k$
factor.  We define the cylinder $C_R(x_0) \subset \R^{n+k}$ by
$$C_r(x) = \{x\in \R^{n+k}\, |\, |\hat{x}-\hat{x}_0|<r,
|\tilde{x}-\tilde{x}_0|<r\}\ .$$
Furthermore, we write $B_r^n(x_0) = \{(\hat{x},\tilde{x}_0)\in
\R^{n+k}\, \ \, |\hat{x}-\hat{x}_0|<r\}$.  

\begin{thm}\label{thm:graph_local} 
Let $(M^n_t)_{0\leq t < T}$ be a smooth mean curvature flow of
embedded $n$-dimensional submanifolds in $\R^{n+k}$ with area
ratios bounded by $D$. Then for any $\eta >0$, there exists $\varepsilon,
\delta >0 $, depending only on $n,k,\eta, D$, such that if $x_0 \in
M_0$ and $M_0\cap C_1(x_0)$ can be written as $\text{graph}(u)$,
where $u:B^n(x_0) \rightarrow \R^k$ with Lipschitz constant less than
$\varepsilon$, then
$$M_t\cap C_\delta(x_0)\qquad t\in [0,\delta^2)\cap [0,T)$$
is a graph over $B^n_\delta(x_0)$ with Lipschitz constant less than
$\eta$ and height bounded by $\eta \delta$.
\end{thm}
 
\begin{rmks}\label{rmks:graph_local}i) In codimension one
  the local estimates of Ecker and Huisken \cite{EckerHuisken91} yield
  that a local bound on the second fundamental form or higher
  derivatives thereof on $M_0\cap C_\delta(x_0)$ imply a corresponding
  bound in $M_t\cap C_{\delta/2}(x_0)$ for $t\in
  [0,\delta^2/4)\cap[0,T)$.\\[0.5ex]
ii) By localizing Huisken's monotonicity formula, see for example
\cite{Ecker04} or \cite{White97}, the result is still true for local mean
curvature flows without an assumption on the area ratios.\\[0.5ex]
iii) The proof of this result uses the local regularity theorem of
White \cite{White05}. By replacing this with Brakke's local regularity
theorem for Brakke flows \cite{Brakke}, see also \cite{Tonegawa14a, Tonegawa14b}, the above statement is still
true if one only assumes initially that $(M^n_t)_{0\leq t<T}$ is an
integer Brakke flow, provided that the flow has no
sudden mass loss in $C_1(x_0)$.
\end{rmks}

{\bf Proof outline.} The cone-like structure at the non-regular
multiple points suggests that the regular evolution out of such a point
should be close to a self-similarly expanding solution. Given such a
non-regular initial network $\gamma_0$ we glue in, around each
non-regular multiple point, a tree-like, self-similarly expanding,
regular solution at scale $s^{1/2}$ to obtain an approximating
network $\gamma_0^s$. Since the curvature of $\gamma_0^s$ is of the
scale $s^{-1/2}$ and the shortest segment of length proportional to $s^{1/2}$ we
obtain from standard short-time existence a solution $\gamma^s_t$ only
up to a time proportional to $s$. 

To show that these solutions exist for a time $T_0>0$ independent of
$s$ we use the local monotonicity formula to show that the
solutions $\gamma_t^s$ are close in an integrated sense to a
self-similarly expanding solution around each non-regular multiple
point. The uniqueness of self-similarly solutions in their
'topological class', together with a compactness argument then yields
that there are many times such that $\gamma_t^s$ is close to the
corresponding self-similarly expanding solution in $C^{1,\alpha}$
around each of the non-regular multiple points. This in turn gives
that the Gaussian density ratios on the appropriate scale are less
than $3/2+\varepsilon$. Theorem \ref{thm:locreg.2} then gives
estimates on the curvature which are independent of $s$, together with
lower bounds on the length of the shortest segment which yields
existence up to a time $T_0$, independent of $s$. Passing to the limit
$s\rightarrow 0$ we obtain the desired solution.

The proof of the local regularity result, Theorem \ref{thm:locreg.2},
follows the proof of White \cite{White05} and the alternative proof Ecker
\cite{Ecker04}. To make this proof work in the case of networks, in a
first part we show  that the only self-similarly shrinking 
networks with Gaussian density less than $\Theta_{S^1}$ are a constant 
line through the origin, or three half-lines meeting at equal angles 
at the origin. The second part is that we show that
any smooth network flow which is weakly close to three half-lines meeting at equal
angles is also smoothly close. To do this we localize the interior
integral estimates in \cite{MNTNetworks}. 

{\bf Structure of the paper.} In section \ref{sect:defi} we give
 the basic definitions. The monotonicity formula for expanding
 solutions is presented in
 section \ref{monotonicity}. In section \ref{sect:self-expanders} we
 give a proof for the uniqueness of self-expanders in their
 topological class, together with a lemma showing that networks which
 are in a weak integral sense close to the self-expander are actually
 $C^{1,\alpha}$-close. Stating the necessary conditions for an
 approximating sequence $\gamma^s_0$ we show the estimates on the
 Gaussian density ratios in section \ref{sec:main theorem}. In section
 \ref{proofs} we give the omitted proofs of some technical lemmas from
 section \ref{sec:main theorem}. Following this, we show that we can
 construct such an approximating sequence by gluing in a
 self-expander at the right scale into $\gamma_0$ in section
 \ref{sec:short time ex} and give the proof of Theorem \ref{thm:short
   time ex}.

In section \ref{sec:loc reg} we first localize the higher order 
integral estimates from \cite{MNTNetworks} and then investigate
tangent flows to the network flow and self-similarly shrinking
solutions. We complete this section by proving the local regularity
result, Theorem \ref{thm:locreg.2}.

In section \ref{sec:graph_local} we prove the pseudolocality result
Theorem \ref{thm:graph_local}.

We finish the paper with an appendix containing several helpful
technical results.

{\bf Acknowledgements:} The third author is grateful to R. Mazzeo and M. S\'aez
for many helpful discussions.

\section{Definitions and set-up}\label{sect:defi}

\begin{defi}[Regular and non-regular network]\label{def:network}
We define a {\em regular}, planar network $\gamma$ as follows.
\begin{itemize}
\item[i)] There is a finite number of points $S=\{a_i\}_{i=1}^n$ on $\gamma$
  such that $\gamma\setminus S$ is a finite union of smooth, embedded
  curves of positive length ({\em branches}).
\item [ii)] If $\sigma$ is a non-compact branch of
  $\gamma$ then it approaches a
  half-line $P$ at infinity, i.e.,
  $$\lim_{R\to\infty}\dist (\sigma\setminus B_R(0), P\setminus
  B_R(0))=0.$$
We will furthermore assume that the curvature of such a non-compact
branch is uniformly bounded.
\item[iii)] Each point in $S$, called a {\em triple point}, is the endpoint of
  three curves 
  $\{\sigma_j\}_{j=1}^3$ satisfying the following
  condition: If $T_j$ denotes the exterior unit 
  tangent vector induced by  each $\sigma_j$, then
  $$T_1+T_2+T_3=0.$$
\end{itemize}
We call a network {\em non-regular} if each point in $S$ is an
endpoint of at least two line segments $\{\sigma_j\}_{j=1}^k\, k\geq 2$, and
the induced exterior unit tangent vectors are mutually distinct
$$ T_i \neq T_j \ \ \text{for}\ i\neq j\, .$$ 
We will call such a point a non-regular multiple point.
\end{defi}

Consider a smooth family of regular, planar
networks $(\gamma_t)_{0\leq t<T}$, i.e. $\gamma_{t_2}$ is a smooth deformation of
$\gamma_{t_1}$. This implies that the number of triple points in
$S_t=\{a_i(t)\}_{i=1}^n$ stays fixed. So we can assume that there 
exists a smooth family of regular parametrizations  $(N_t)_{0\leq t<T}$
of the evolving network. We will call $(\gamma_t)_{0\leq t<T}$ a solution
to the {\em network flow} if the deformation vector  
$$\frac{dN}{dt}=X\quad\quad\mbox{satisfies}\quad\quad X^{\bot}=\vec{k}$$
at each non-singular point.  		

\begin{rmks}
  i) As a consequence, using the above notation, at each triple point
$$\sum_{j=1}^3\langle\vec k_j, J T_j\rangle=\langle X,JT_1+JT_2+JT_3\rangle=0,$$
where $J$ is the complex structure.\\[0.5ex]
  ii) Note that for a network without triple points, i.e.\! a union of
  curves, this is curve shortening flow.\\[0.5ex]
  iii) A network flow still satisfies the avoidance principle when
  comparing to smooth solutions of curve shortening flow. Comparing
  with big shrinking circles it is easy to see that any half-line $P$ at
  infinity will remain fixed under the flow.\\[0.5ex]
  iv) Note that ii) in Definition \ref{def:network} implies
  that such a network has bounded length ratios, i.e. there exists
  $D>0$ such that
$$ \H^1(\gamma \cap B_r(0)) \leq D r\ . $$
v) By the work of Mantegazza, Novaga and Tortorelli \cite{MNTNetworks,
  Mantegazza04} it is known that
 for a given smooth, regular, planar network a smooth solution to the
 network flow exists,
 at least for a short time, provided it is compact with possible fixed
 endpoints. It is shown there that the solution
 exists as long as the curvature of the evolving network stays
 bounded, and none of the lengths of the branches goes to zero. This
 statement can be easily extended to the case of regular networks with
 non-compact branches as in Definition \ref{def:network}, see the
 beginning of section \ref{sec:short time ex}. \\[0.5ex]
vi) It would also be possible to study networks with fixed or
moving endpoints. To avoid the non-conceptual, but technical
difficulties arising from the contribution of the endpoints, we do
not consider this case. 
\end{rmks}
\begin{defi}\label{def:c1alphaclose}	
Let $\chi$ be a regular network with finitely many triple points. We
say that $\chi$ is of class $C^{k,\alpha}$ where $ k\geq 1,\, 0\leq
\alpha \leq 1$ if there exists $\delta>0$ and a collection of points
$(p_i) \subset \chi$, either finitely or countably many,  such that
\begin{itemize}
\item[a)] the collection of balls $(B_{3\delta/4}(p_i))$ covers $\chi$,
\item[b)] each ball $B_\delta(p_i)$ contains at most one triple
  point. If it contains no triple point, then $B_{3\delta/4}(p_i)\cap\chi$ can
  be written as a graph over its affine tangent line at $p_i$, where
  the graph function has $C^{k,\alpha}$-norm less than one.
\item[c)] If $B_\delta(p_i)$ contains a triple point, then $p_i$ is
  the triple point, and $\chi \cap  B_\delta(p_i)$ consists of three
  curves meeting at $p_i$. Each of the curves in $B_{3\delta/4}(p_i)$
  can be written as a graph over the corresponding affine tangent
  half-line at $p_i$ where
  the graph function has $C^{k,\alpha}$-norm less than one.
\end{itemize}
We say that another regular network $\sigma$ is $\ve$-close to
$\chi$ in $C^{k,\alpha}$, if $\sigma$ is contained in the $\delta/2$-neighborhood of
$\chi$ and the triple junctions of $\sigma$ are in one to one
correspondence with the triple junctions of $\chi$, with the triple
junctions of $\sigma$ being in a $\delta/2$-neighborhood of the triple
junctions of $\chi$. Furthermore, in case b) in the above local graph
representation, $\sigma \cap B_{3\delta/4}(p_i)$ can be written as
a graph as well, where the difference of the graph functions is less
than $\ve$ in $C^{k,\alpha}$. In case c) we assume that there exist unit vectors $N_i$ such that
$\sigma \cap B_{3\delta/4}(p_i)$ can be written as $\chi + u_iN_i$,
where the $u_i$ are defined on a connected sub-domain of $B_\delta(p_i)
\cap \chi$ and continuous. Restricted to each of the three local
branches of $\chi$ we assume that the  $C^{k,\alpha}$-norm of $u_i$ is
less than $\ve$ with respect to arc-length parametrization on each branch. 
\end{defi}

\section{Monotonicity formulas}\label{monotonicity}

Let $\tilde{\theta_t}$ be the angle that the tangent vector of
$\gamma_t$ makes with the $x$-axis. 
This is  a well defined function up to multiple of $\pi$ away from the triple junction
points. Because at each of these points 
the angle jumps by $2\pi/3$, there is a well defined function
$\theta_t$ which is continuous on $\gamma_t$ and  coincides with
$\tilde{\theta_t}$ up to a multiple of $\pi/3$. An important
observation is that $\vec k=J\nabla \theta_t$, where $J$ is the 
complex structure.

Set $\lambda=xdy-ydx$. We assume that the planar network $\gamma_t$
has no loops, so we can define $\beta_t$ to be such that $$d\beta_t=\lambda_{\gamma_t}.$$
Note that $\beta_t$ is Lipschitz because its gradient is 
bounded linearly and thus $\beta_t$ grows at most quadratically. 

Finally, for any $x_0\in \R^2$  and $t_0$, define for $t<t_0$ the
backwards heat kernel centered at $(x_0,t_0)$:
\begin{equation}\label{eq:def.0}
\rho_{x_0,t_0}(x,t)=\frac{1}{\sqrt{4\pi(t_0-t)}}e^{-\frac{|x-x_0|^2}{4(t_0-t)}}\, .
\end{equation}

\begin{lemm} The following evolution equations hold away from the triple junction points:
  \begin{align}
  \tag{i}    \frac{d\theta_t}{dt} &=\Delta\theta_t+\langle\nabla 
    \theta_t, X \rangle;\\
  \tag{ii}
    \frac{d\beta_t}{dt} &=\Delta
    \beta_t+\langle\nabla \beta_t, X \rangle-2\theta_t;\\
  \tag{iii} \begin{split}
    \frac{d\rho_{x_0,t_0}}{dt} &=
    -\Delta\rho_{x_0,t_0}+\langle\nabla
    \rho_{x_0,t_0}, X \rangle\\ 
     &\qquad \qquad -\left|\vec
      k+\frac{(x-x_0)^{\bot}}{2(t_0-t)}
    \right|^2\rho_{x_0,t_0}+|\vec k|^2\rho_{x_0,t_0}.
    \end{split}
 \end{align}
\end{lemm}
\begin{proof}
  The derivation of these equations proceeds as in the
  smooth case except that now we have a tangential
  motion that needs to be taken into account. For
  this reason we will only show the second formula.

  The family of functions $\beta_t$ can be  chosen so
  that its time derivative is continuous. Then, denoting
  by $\mathcal{L}_{X}$ the Lie derivative  in the $X$
  direction, we obtain from Cartan's formula
  $$d(d\beta_t/dt)=\mathcal{L}_{X}\lambda=d(X\lrcorner\lambda)+
  X\lrcorner d\lambda=d(X\lrcorner\lambda)-2d\theta_t,$$
  where in the last equality we use the fact that 
  $$(X\lrcorner d\lambda)_{\gamma_t}=(X^{\bot}\lrcorner
  d\lambda)_{\gamma_t}
  =(J\nabla \theta_t\lrcorner d\lambda)_{\gamma_t}=-2d\theta_t.$$
  Therefore $$d(d\beta_t/dt+2\theta_t-X\lrcorner\lambda)=0.$$
  Note that the function which has differential zero is
  continuous on $\gamma_t$ and so we can add a time
  dependent constant to each $\beta_t$ to obtain that
  $$\displaystyle \frac{d\beta_t}{dt}=\langle X,Jx\rangle-2\theta_t.$$
  The desired formula follows from 
  $\langle X,Jx\rangle=\Delta \beta_t+\langle X,\nabla \beta_t\rangle$.
  \end{proof}

\begin{lemm}\label{evol} Let $f$ be in $C^2(\R)$. The following identities hold:\\[-1ex]
  \begin{align}
  \tag{i} 
    \begin{split}\frac{d}{dt}\int_{\gamma_t} f(\theta_t) 
      \rho_{x_0,t_0}\, d\mu &=-\int_{\gamma_t}f''(\theta_t)|\vec
      k|^2\rho_{x_0,t_0}\, d\mu\\
      & -\int_{\gamma_t}f(\theta_t)\left|\vec
        k+\frac{(x-x_0)^{\bot}}{2(t_0-t)}\right|^2
      \rho_{x_0,t_0}\, d\mu.
    \end{split}
     \end{align}
\begin{align}
  \tag{ii}
    \begin{split}
      \frac{d}{dt}\int_{\gamma_t}
      f(\beta_t+2t\theta_t)&\rho_{x_0,t_0}\,  
      d\mu =\\& -\int_{\gamma_t}f''(\beta_t+2t\theta_t)|x^{\bot}-2t\vec
      k|^2\rho_{x_0,t_0}\, d\mu\\
      & -\int_{\gamma_t}f(\beta_t+2t\theta_t)\left|\vec
        k+\frac{(x-x_0)^{\bot}}
        {2(t_0-t)}\right|^2\rho_{x_0,t_0}\, d\mu.
    \end{split}
  \end{align}
\end{lemm}
\begin{proof}
  We prove the second identity. Set $\alpha_t=\beta_t+2t\theta_t$. Then
  $$\frac{df(\alpha_t)}{dt}=\Delta f(\alpha_t)-f''(\alpha_t)|\nabla 
  \alpha_t|^2+\langle\nabla f(\alpha_t),X\rangle$$ 
  and so
  \begin{multline*}
    \frac{d}{dt}\int_{\gamma_t}
    f(\alpha_t)\rho_{x_0,t_0}\, d\mu=\int_{\gamma_t}
    \rho_{x_0,t_0}\Delta f(\alpha_t)-f(\alpha_t)\Delta
    \rho_{x_0,t_0}\, d\mu\\
    +\int_{\gamma_t}\d (X^{\top} f(\alpha_t)\rho_{x_0,t_0})d\mu 
    -\int_{\gamma_t}f''(\alpha_t)|x^{\bot}-2t\vec k|^2\rho_{x_0,t_0}\,
    d\mu\\
    -\int_{\gamma_t}f(\alpha_t)\left|\vec k+\frac{(x-x_0)^{\bot}}
      {2(t_0-t)}\right|^2\rho_{x_0,t_0}d\mu.
  \end{multline*}
  We need to show that the first two integral terms vanish. Decompose
  $\gamma_t$ into $k$ smooth curves
  $\{\sigma_j\}_{j=1}^k$. Then, we obtain from Green's formulas
  \begin{multline*}
    \int_{\gamma_t}\rho_{x_0,t_0}\Delta f(\alpha_t)-f(\alpha_t)\Delta \rho_{x_0,t_0}\,d\mu\\
    =\sum_{j=1}^k \int_{\partial \sigma_j} \rho_{x_0,t_0} \langle \nabla
    f(\alpha_t), 
    \nu\rangle-f(\alpha_t)\langle\nabla \rho_{x_0,t_0}, \nu\rangle\, d\mu,
  \end{multline*}
  where $\nu$ is the exterior unit normal to each $\sigma_j$. It is  
  straightforward to see that if $\sigma_j$ is non-compact then the 
  boundary term at ``infinity'' vanishes. Let  $x_1$ be a triple junction
  point meeting three line segments, which we relabel as $\sigma_1,
  \sigma_2,$ and $\sigma_3$.  Then, at the point $x_1$
  \begin{equation*}\begin{split}
  \sum_{i=1}^3 \langle \nabla f(\alpha_t), \nu_i\rangle &=\sum_{i=1}^3
  f^\prime\langle Jx_1-2tJ\vec k_i, \nu_i\rangle\\ &=
  f^\prime\langle J x_1-2tJ X, \nu_1+\nu_2+\nu_3\rangle=0.
  \end{split}
  \end{equation*}
  The same argument shows that
  $$
  \sum_{i=1}^3 \langle \nabla \rho_{x_0,t_0}, \nu_i\rangle=0
  $$
  and so $$\int_{\gamma_t}\rho_{x_0,t_0}\Delta
  f(\alpha_t)-f(\alpha_t)\Delta \rho_{x_0,t_0}\, d\mu=0.$$
  For the second integral term we use again the decomposition
  $$\int_{\gamma_t}\d (X^{\top} f(\alpha_t)\rho_{x_0,t_0})\,d\mu
  =\sum_{j=1}^k\int_{\sigma_j}\d (X^{\top} f(\alpha_t)\rho_{x_0,t_0})\,d\mu$$
  and one can argue as before to conclude that 
  $$\int_{\gamma_t}\d (X^{\top} f(\alpha_t)\rho_{x_0,t_0})\,d\mu=0.$$
\end{proof}

In the applications later, the evolving network will only be {\it
  locally} tree-like, i.e.\! only locally without loops. To apply the above
monotonicity formula we have to localize it. We assume that
$(\gamma_t)_{0\leq t<T}$ is a smooth solution to the network flow such
that $\gamma_t\cap B_4$ does not contain any closed loop for all
$0\leq t<T$. As before, we define $\beta$ locally on $\gamma_t\cap
B_4$.

Let $\varphi$ be a smooth cut-off function such that $\varphi = 1 $ on
$B_2$, $\varphi = 0$ on $\R^2\setminus B_3$ and $0\leq
\varphi\leq 1$.

\begin{lemm}\label{thm:evol.local} The following estimate holds:
\begin{equation*}\begin{split}\frac{d}{dt}\int_{\gamma_t}\varphi
|\beta_t+2t\theta_t|^2\rho_{x_0,t_0}\, d\mu &\leq -
\int_{\gamma_t} \varphi |x^\perp - 2t\vec{k}|^2
\rho_{x_0,t_0} d\mu \\
&\ \ \  + C \int_{\gamma_t\cap(B_3\setminus B_2)}
|\beta_t+2t\theta_t|^2\rho_{x_0,t_0}\, d\mu\, .
\end{split}\end{equation*} 
\end{lemm}
\begin{proof}
  We have
\begin{equation*}\begin{split}
\Big(\frac{d}{dt}-\Delta_{\gamma_t}\Big) \varphi &= -
\Delta_{\R^2}\varphi + D^2\varphi(\nu,\nu) + \langle \nabla \varphi,
X^T \rangle\\
&\leq C \chi_{B_3\setminus B_2} +  \langle \nabla \varphi,
X^T \rangle\, .
\end{split}
\end{equation*}
As in the proof of Lemma \ref{evol} we set $\alpha_t =
\beta_t+2t\theta_t$. Then
\begin{equation*}
  \begin{split}
    \Big(\frac{d}{dt}-\Delta\Big) \varphi \alpha_t^2&\leq
    -2\langle \nabla \varphi, \nabla \alpha_t^2\rangle - 2 \varphi
    |\nabla \alpha_t|^2 + \langle \nabla (\varphi \alpha_t^2),
    X^T\rangle \\
    &\quad\ + C\chi_{B_3\setminus B_2}\alpha_t^2\\
    &\leq - \varphi
    |\nabla \alpha_t|^2 + \langle \nabla (\varphi \alpha_t^2),
    X^T\rangle + C\chi_{B_3\setminus B_2}\alpha_t^2\ ,
  \end{split}
\end{equation*}
where we estimated
\begin{equation*}
  \begin{split}
|\langle \nabla \varphi, \nabla \alpha_t^2\rangle| &\leq 2 |D\varphi|
|\alpha_t| |\nabla \alpha_t| \leq 
4\frac{|D\varphi|^2}{\varphi} \alpha_t^2+ \frac{1}{2} \varphi
|\nabla\alpha_t|^2\\
&\leq C\chi_{B_3\setminus B_2} \alpha_t^2+ \frac{1}{2} \varphi
|\nabla\alpha_t|^2\, .
 \end{split}
\end{equation*}
The rest follows as in the proof of the previous lemma.
\end{proof}

\section{Uniqueness of self-expanders}\label{sect:self-expanders}

Consider the negatively curved metric 
$$g=\exp({|x|^2})(dx_1^2+dx_2^2).$$
A network $\psi$ is said to be a {\em geodesic} for $g$ if, when
parametrized  proportionally to arc-length, is  a critical point for
the  length functional when restricted to variations with compact support. The network is said to be a {\em self-expander} if
$\vec{k}=\psi^{\bot}$ on each branch and we say that the self-expander is
{\em regular} if it has only triple junctions and the angles at each triple
junction are $2\pi/3$.

Given a function $u$ or a curve $\psi$, we denote by $u', u'', \psi',$ and $\psi''$, the correspondent derivatives with respect to the space parameter.
 
In this section we show that regular self-expanders are unique in their topological class.
 
\begin{lemm}[Ilmanen and White] 
  A  network $\psi$  is a regular self-expander  if and only if it is a geodesic for $g$.
\end{lemm}

\begin{proof}
  Let $(\psi_s)_{0\leq s\leq \varepsilon}$ be a compactly supported
  continuous deformation of $\psi$ which is a $C^1$ deformation when
  restricted to each branch. Each network $\psi_s$ has only triple
  junctions. If we set
  $$X=\frac{d\psi_s}{ds}\quad\mbox{and}\quad
  T=\frac{\psi_s'}{|\psi_s'|}$$ then, assuming parametrization
  proportional  to arc-length, we have  for  each branch
  \begin{multline}\label{first-derivative}
    \frac{d}{ds}\int_a^b (g(\psi'_s,\psi'_s))^{1/2}dt  =
    \frac{d}{ds}\int_a^b \exp({|x|^2}/2)|\psi'_s|dt \\
    =\left(\langle T, X\rangle
      \exp({|x|^2}/2)\right]_a^b+\int_a^b\langle x^{\bot}-\vec{k},X
    \rangle|\psi'_s|\exp(|x|^2/2)dt.
  \end{multline}
  If $\psi$ is a geodesic then by choosing variations $X$ compactly supported on each branch we obtain that indeed $ x^{\bot}=\vec{k}$ on each branch. Choosing variations supported on a neighborhood of each triple junction we obtain that 
  $$T_1+T_2+T_3=0,$$
  where $T_1, T_2, T_3$ denote the outward unit tangent vectors at each triple junction and so $\psi$ is a regular self-expander. Likewise, if $\psi$ is a regular self-expander, it is simple to see that it is a critical point for the length functional.
\end{proof}

\begin{defi}\label{def:asymptotic}
  We say that a self-expander $\psi$ has an end asymptotic to a
  half-line $L=\{xe^{i\alpha}\,|\,x\geq 0\}$ if, for $R$ large enough,
  a connected component $\bar \psi$ of $\psi\setminus B_R$ can be
  parametrized as
  $$\bar \psi= \{xe^{i\alpha}+u(x)e^{i(\alpha+\pi/2)}\,|\, \mbox{ for
    all }x\geq R\},\quad\mbox{ where }\lim_{x\to\infty}u(x)=0.$$
\end{defi}

\begin{lemm}\label{decay}
  Let $P$ be a union of half-lines meeting at the origin and $\psi$ a self-expander for which
  $$\lim_{r\to\infty}\dist(\psi\setminus B_r, P)=0.$$
  Then $\psi$ is asymptotic to $P$ in the sense of Definition \ref{def:asymptotic}. Moreover, the decay of $u$ for $ x \geq R$ is given by
  $$|u|\leq C_1e^{-x^2/2}, \quad |u'|\leq C_1 x^{-1}e^{-x^2/2},\quad  |u''|\leq C_2  e^{-x^2/2},$$
  and
  $$|u^{(3)}|\leq C_3 xe^{-x^2/2}, \quad |u^{(4)}|\leq C_4 x^2e^{-x^2/2},$$
  where each $C_i$ depends only on $R$, $u(R)$ and $u'(R)$.
\end{lemm}
\begin{proof}
  In \cite{SchnuererSchulze_2007} it is shown that each asymptotic end
  of a self-expander is asymptotic to a half-line. Even more the graph function $u$
  decays exponentially. 
 Since $\psi$ is a self-expander the function $u$
  satisfies
  $$ u^{\prime\prime} = (1+(u^\prime)^2)(u-xu^\prime).$$
  By possibly changing orientation, a simple application of the maximum principle  (see
  \cite{SchnuererSchulze_2007}) implies that we can assume without loss of generality
  that $u > 0$ and $u^\prime < 0$, if the self expander is not identical with
  the half-line. The function $v= u-xu^\prime$ is strictly positive
  and satisfies 
  $$ v^\prime = -x (1+(u^\prime)^2)v < -x v\ .$$
  Integrating this inequality yields the first two estimates. Inserting that into the equation for  $u''$ we get the third estimate and the fourth  and fifth estimates come from computing $u^{(3)}$, $u^{(4)}$, and using the previous derived estimates.
\end{proof}

We say that two self-expanders $\psi_0$ and $\psi_1$ are {\em
asymptotic to each other} if their ends are asymptotic to the 
same half-lines. In this setting, we say they have the {\em same topological class} if there is a smooth family of maps
$$F_t:\psi_0 \longrightarrow\R^2,\quad 0\leq t\leq 1$$
such that $F_0$ is the identity, $F_1(\psi_0)=\psi_1$, the distance between any two triple junctions of $F_t(\psi_0)$ is uniformly bounded below, and
$$\lim_{r\to\infty}\sup\left\{\left|{d F_t}\right|_{F_t(x)}\,|\, x\in \gamma_0\setminus B_r(0)\right\}=0\mbox{ for every }0\leq t\leq 1.$$

\begin{thm}\label{thm:unique}
  If $\psi_0$ and $\psi_1$ are two regular  self-expanders asymptotic
  to each other 
  and in the same topological class, then they coincide.
\end{thm}
\begin{proof}
  Let $(x^0_i)_{i\in A}$ and $(x^1_i)_{i\in A}$ denote the  triple
  junctions (finite set) of $\psi_0$ and $\psi_1$
  respectively.  Because the networks are in the same topological class, we can rearrange the elements of  $(x^0_i)_{i\in A}$ so that each $x_i^0$ is connected to $x_i^1$ by the existing deformation of $\psi_0$ into $\psi_1$. Denote by $(x^s_i)_{0\leq s \leq 1}$ the unique geodesic connecting these points.  
  
  For each $s$  we consider the  network $\psi_s$ such that  if $x_i^0$ is connected to $x_j^0$ by a geodesic, then $x_i^s$ is connected to $x_j^s$ through a geodesic as well.   To handle the non-compact branches we proceed as follows. Let $P$ denote a common asymptotic half-line to $\psi_0$ and $\psi_1$, which means that there are geodesics $\gamma_0\subset \psi_0$, $\gamma_1\subset \psi_1$ asymptotic to $P$ at infinity and starting at some points $x^0_i$ and $x^1_i$ respectively. Define $\gamma_s$ to be the unique geodesic starting at $x^s_i$ and asymptotic to $P$. Because these are geodesics with respect to a negatively curved metric it is easy to see that if $\gamma_s$ intersects $\gamma_{s'}$ then they must coincide.
  
  Hence, we  have constructed a smooth family  of
  triple-junction networks $(\psi_s)_{0\leq s\leq 1}$ connecting 
  $\psi_0$ and $\psi_1$ and such that:
\begin{enumerate}

\item [i)] The triple-junctions $(x^s_i)_{i\in A}$ of $\psi_s$ connect
 the triple-junctions of $\psi_0$ to the ones of $\psi_1$ and, for
 each index $i$ fixed, the path $(x_i^s)_{0\leq s \leq1}$ is a
 geodesic with respect to the metric $g$.\\

 \item [ii)] Each branch of $\psi_s$ is a geodesic for $g$.\\

\item[iii)]There is $R$ large enough so that $\psi_s\setminus B_R(0)$ has
  $N$ connected components, each asymptotic to an half-line $L_j$,
  $j=1,\ldots, N$. We can find angles $\alpha_j$  such that each end
  of $\psi_s$ becomes parametrized as  
$$\psi_s(x)=xe^{i\alpha_j}+u_{j,s}(x)e^{i(\alpha_j+\pi/2)}\quad\mbox{for all }x\geq R.$$
This follows from Lemma \ref{decay}.\\

\item[iv)] The vector $$X=\frac{d}{ds}\psi_s$$ is continuous,  $C^1$
  when restricted to each branch, and 
  $$X=O(e^{-r^2/2}), \quad \nabla X=O(r^{-1}e^{-r^2/2}).$$
  Moreover
  $$\alpha_{j,s}=\frac{d u_{j,s}}{ds}$$
  satisfies
  $$|\alpha_{j,s}|=O(e^{-x^2/2})\quad  |\alpha'_{j,s}|=O(x^{-1}e^{-x^2/2}).$$
  
   It is enough to provide justification for the second set of estimates.
  For ease of notation we omit the indexes $s$ and $j$ on $\alpha_{j,s}$ and $u_{j,s}$. We have
  $$\alpha''=(1+(u')^2)(\alpha-x\alpha')+2u'\alpha'(u-xu').$$
 We can assume without loss of generality that $\alpha(R)\geq 0$. Moreover, it follows from our construction that 
 $$\lim_{x\to\infty}|\alpha(x)|+|\alpha'(x)|=0.$$
 A simple application of the maximum principle shows that $\alpha$ can not have negative local minimum or a positive  local maximum. Hence, $\alpha\geq 0$ and $\alpha'\leq 0$. The function $\beta=\alpha-x\alpha'$ satisfies
 $$\beta'=-x(1+(u')^2)\beta-2xu'\alpha'(u-xu')\leq -x\beta$$
because $u'(u-xu')\leq 0$ (see proof of Lemma \ref{decay}), and integration of this inequality implies property iv).

\end{enumerate}

Denote by $L$ the length function with respect to the metric $g$ and
consider the family of functions
\begin{multline*}
  F_t(s)=L(\psi_s\cap
  B_{2R}(0))+\sum_{j=1}^N\int_{2R}^t\exp((x^2+u^2_{j,s})/2)\sqrt{1+(u_{j,s}')^2}dx\\
  -N\int_{2R}^t\exp(x^2/2)dx.
\end{multline*}
The decays given in Lemma \ref{decay} imply the existence of a constant $C$ such that for every $t\leq t'$
\begin{equation}\label{C^3}
\|F_t-F_{t'}\|_{C^3}\leq C\exp(-t),
\end{equation}
and so when $t$ tends to infinity $F_t$ converges uniformly in $C^2$ to a function $F$. Furthermore, if $s=0$ or $s=1,$ we have  from combining \eqref{first-derivative} with property iv) that
$$ \lim_{t\to\infty} \frac{d F_t}{ds}(s)=0,$$
and thus $F$ has a critical point when $s=0$ or $s=1$.

A standard computation shows that on each  compact branch we have (assuming
parametrization proportional to arc-length)
\begin{equation*}\begin{split}
 \frac{d^2}{ds^2}\int_a^b (g(\psi'_s,\psi'_s))^{1/2}&dt = \int_a^b|\psi'_s |^{-1}(|(\nabla _{\psi_s'}
  X)^{\bot}|^2\\
  & \quad -\mbox{Rm}(X,\psi'_s,\psi'_s,X))dt +\left(|\psi'_s|^{-1}
    g(\nabla_X X, \psi'_s)\right]_a^b\\
  &=\int_a^b|\psi'_s |^{-1}(|(\nabla _{\psi'_s} X)^{\bot}|^2-\mbox{Rm}(X,\psi'_s,\psi'_s,X))dt,
\end{split}\end{equation*}
where we used  property i) on the second equality and all the
geometric quantities are computed with respect to the metric $g$. Combining this identity with property iv) we have
\begin{equation*}
\frac{d ^2 F_t}{ds^2}=\int_{\psi_s\cap B_{t}(0)}|\psi'_s |^{-2}(|(\nabla _{\psi'_s} X)^{\bot}|^2-\mbox{Rm}(X,\psi'_s,\psi'_s,X))dl+O(e^{-t}).
\end{equation*}
The Gaussian curvature of $g$ is equal to $-e^{-|x|^2}$ and so the integrals above are bounded independently of $t$. Therefore, we obtain from \eqref{C^3} that
$$\frac{d^2 F}{ds^2}(s)= \int_{\psi_s}|\psi'_s |^{-2}(|(\nabla _{\psi'_s} X)^{\bot}|^2-\mbox{Rm}(X,\psi'_s,\psi'_s,X))dl\geq 0$$
where the last inequality comes form the fact that  $g$ has strictly negative Gaussian curvature. As a result,   $F$ is a convex function with  two critical points and hence identically constant.
The formula above  implies that $X$ must be
a constant multiple of $\psi'_s$ and thus it must vanish at all 
triple-junction  points.  The fact that $X$ is continuous implies that $X$ is identically zero and this proves the desired result. 
\end{proof}

Before using this Theorem to prove a compactness result we need one more definition.

\begin{defi}
Two regular networks $\sigma_0$ and $\sigma_1$ are in the same $(\nu,\eta, r,R,C)$ topological class if there is a smooth family $(\hat \sigma_t)_{0\leq t\leq 1}$ of networks, with possible boundary points, such that for every $0\leq t\leq 1$
  \begin{enumerate}
  \item [a)]  the distance between any two triple junctions of $\hat \sigma_t$ is bigger or equal to $\eta$;
  \item[b)] all the triple junctions of $\hat \sigma_t$ are contained in $B_{r}(0)$ and the boundary points of $\hat \sigma_t$ are contained outside $B_{R}(0)$, with $r\leq R$;
  \item[c)] For every $R \geq s\geq r$ 
  $$\dist (\hat \sigma_t\setminus B_s, P)\leq
    \nu+C\exp(-s^2/C);$$
 \item[d)]$\sigma_0\cap B_{R} \subseteq \hat \sigma_0$ and $\sigma_1\cap B_{R} \subseteq \hat \sigma_1$.
  \end{enumerate}

\end{defi}

We can now state the following corollary.
 
\begin{cor}\label{unique}
  Let $\psi$ be a regular self-expander with ends asymptotic to a
  union of half-lines $P$. Fix $r_1$, $\eta$, $C_1$,  $D_1$,  $\alpha<1/2$,  and $R$.
  
  For every $\varepsilon$, there are
  $R_1\geq R$,
  $\beta$, and $\nu$, all depending on $\varepsilon,$ $r_1,$ $\eta,$ $C_1,$ $D_1,$ $\alpha,$ $P,$ $R,$, so that   if $\sigma$ is a regular network   that satisfies:
  \begin{enumerate}
  \item[i)] $$\displaystyle \H^1(\sigma\cap B_r(x))\leq D_1r\mbox{ for all $x$  and $r$};$$
  \item[ii)]$$\int_{\sigma\cap B_{R_1}(0)}|\vec k-x^{\bot}|^2d\H^1\leq \beta;$$
  \item[iii)] $\sigma$ and $\psi$ are in the same  $(\nu,\eta_1, r_1,R_1,C_1)$ topological class
    \end{enumerate}
  then $\sigma$ must be  $\varepsilon$-close in $C^{1,\alpha}(B_{R_1}(0))$ to $\psi$.
\end{cor}
\begin{proof}
We start by finding $R_1\geq R$ and $\nu$ so that if $\sigma$ is a regular self-expander  in the same $(2\nu,\eta/2, r_1+1,R_1-1,C_1)$ topological class as $\psi$ then $\sigma$ must be  $\varepsilon/2$-close in $C^{1,\alpha}(B_{R_1}(0))$ to $\psi$.
	
	Suppose not. Then we can find a sequence of self-expanders $\sigma_i$  with $R_i$ tending to infinity, $\nu_i$ tending to zero, and such that $\sigma_i$ is not $\varepsilon/2$-close in $C^{1,\alpha}(B_{R_i}(0))$ to $\psi$.  Let $b_i$ denote a smooth branch of $\sigma_i$ that connects $\sigma_i\cap \{|x|=R_i\}$ to one of the triple junctions inside $B_{2r_1}(0)$. Because of c) and Lemma \ref{decay}, there is some $r_2$ such that, for every $i$ large enough,  $b_i\setminus B_{r_2}(0)$ can be written as a graph of a function with $C^{1,\alpha}$ norm less than $\varepsilon/4$ and defined over part of $P\setminus B_{r_2}$. As a result, if $\sigma_i$ is not $\varepsilon/2$-close in $C^{1,\alpha}(B_{R_i}(0))$ to $\psi$,  we can find $r_3$ such that $\sigma_i$ is not $\varepsilon/2$-close in $C^{1,\alpha}(B_{r_3}(0))$ to $\psi$ for every $i$ large enough. Because each branch of $\sigma_i$ is a geodesic of $g$, it is simple to see that we have uniform length bounds for $\sigma_i$. Standard compactness arguments show that a subsequence of $\sigma_i$ converges in $C^{1,\alpha}$ to regular self-expander $\sigma$ which, in virtue of property c), is asymptotic to $P$ at infinity. If we can show that $\psi$ and $\sigma$ are in the same topological class, then Theorem \ref{thm:unique} implies that they have to coincide and this is a  contradiction. 
	
	Arguing  as in Theorem \ref{thm:unique}, we can change the family $(\hat \sigma^i_t)_{0\leq t\leq 1}$ of networks given by hypothesis iii) and construct a  family $(\hat \psi^i_t)_{0\leq t\leq 1}$ of networks connecting $\psi \cap B_{R_i-1}(0)$ to  $\sigma_i \cap B_{R_i-1}(0)$ such that all the branches are geodesics for $g$, and those which intersect $\{|x|=R_i-1\}$ have a uniform decay towards the half-lines of $P$ ( $\hat \psi^i_t$ should  satisfy, with obvious modifications, properties i)-iv) described  in the proof of Theorem \ref{thm:unique}). Making $i$ tending to infinity,  it is easy to recognize that $(\hat \psi^i_t)_{0\leq t\leq 1}$ converges to a family of networks $(\hat \psi_t)_{0\leq t\leq 1}$ connecting $\psi$ to $\sigma$  and satisfying properties  i)-iv) mentioned in Theorem \ref{thm:unique}. Hence the self-expanders must be in  the same topological class.

	Set $\varepsilon_1=\min\{\varepsilon/2, \nu, \eta/2, r_1, 1/2\}$, and  let $\sigma$ be a regular network  satisfying the hypothesis of the lemma.
 Condition i) and ii) imply that
  $$ \int_{B_{R_1}(0)\cap\sigma} |\vec{k}|^2 d\H^1 \leq \beta+\int_{B_{R_1}(0)\cap
    \sigma}|x^{\bot}|^2d\H^1\leq \beta +D_1 R_1^3 .$$
  Thus we have uniform $C^{1,1/2}$ estimates for $\sigma$ in $B_{R_1}(0)$. A standard compactness argument shows  that by taking
  $\beta$ small enough, we can assume that $\sigma$ is
  $\varepsilon_1$-close in $C^{1,\alpha}$ to a regular self-expander $\psi^\prime$
  in $B_{R_1}(0)$. By the reasoning before, we thus get that $\psi^\prime$ is $\varepsilon/2$-close in $C^{1,\alpha}$ to $\psi$ in $B_{R_1}(0) $ and this implies the desired result. 
\end{proof}
 
\section{Main Theorem}\label{sec:main theorem}
To show the short-time existence result for non-regular initial
networks, we will use a special family of approximating regular
networks. We will state the needed properties of such an approximating
family below and show in the sequel the needed estimates for the proof
of the short-time existence result. We will show in section
\ref{sec:short time ex} that for any non-regular initial network such
an approximating family exists.

Fix a regular self expander $\psi$ which is asymptotic to a union of half-lines denoted by $P$. Note that, by Lemma \ref{decay}, $P$ coincides with the blow-down of $\psi$. For any $x_0\in \R^2$  and $t>0$  denote 
$$\Phi(x_0,t)(x)= \rho_{x_0,0}(x,-t)=\frac{1}{\sqrt{4\pi t}}\exp\left({-\frac{|x-x_0|^2}{4t}}\right).$$ 
We also use the notation
$$A(r,R)=\{x\in \R^2\,:\,r\leq |x|\leq R\}.$$
Let $(\gamma^{s})_{0< s\leq c}$ be a family of regular networks on
$\R^2$ such that for every $0< s\leq c\,$:
\begin{itemize}
\item[H1)] There is a constant $D_1$ such that
  $$\H^1(\gamma^s\cap B_r(x))\leq D_1r\mbox{ for all $x$ and $r$}.$$
\item[H2)] There is a constant $D_2$ such that for every $s$ and $x$ in $\gamma^s$
  $$|\theta^s(x)|+|\beta^s(x)|\leq D_2(|x|^2+1).$$
\item[H3)] $\tilde\gamma^s=\displaystyle \frac{\gamma^s}{\sqrt{2s}}$
  converges in $C^{1,\alpha}_{\l}$ to $\psi$. Without loss of generality we assume that
  \begin{equation*}\label{zero}
    \lim_{s\to 0}(\theta^s+\tilde\beta^s)=0,
  \end{equation*}
where $\tilde\beta^s$ is primitive for the Liouville form of $\tilde \gamma^s$.
\item[H4)] The connected components of
  $P\cap A(r_0\sqrt{s},4)$ are in one-to-one correspondence with the
  connected components of $$\gamma^s\cap A(r_0\sqrt{s},4)$$ and if
  $\theta$ is the angle that a half-line in $P$ makes with the
  $x$-axis, there is a function $u_s$ such that a connected
  component $\sigma$ of $\gamma^s\cap A(r_0\sqrt{s},4)$ can be
  parametrized as
  $$\sigma= \{xe^{i\theta}+u_s(x)e^{i(\theta+\pi/2)}\,|\, \mbox{ for
    all }  r_0\sqrt s\leq x\leq 4\}.$$
  Moreover, the function $u_s$ satisfies
  $$|u_s(x)|+ |x|\left|\frac{du_s}{dx}\right|+|x|^2\left|\frac{d^2u_s} {dx^2}\right| \leq D_3 \left(|x|^2+(2s)^{1/2}\exp\left(-{|x|^2}/{4s} \right) \right)$$
  for some constant $D_3$.

\end{itemize}

Assume that $(\gamma^s_t)_{t\geq 0}$ is a smooth solution to network
flow with initial condition $\gamma^s$ and denote by $\Theta^s_t(x,r)$
the Gaussian density of $\gamma^s_t$
\begin{equation}
\Theta^s_t(x_0,r)=\int_{\gamma^s_t}\Phi(x_0, r^2)d\H^1.
\end{equation}
Note that in our previous notation we have $\Theta_t^s(x_0,r) =
\Theta(x_0, t+r^2, r)$ with respect to the flow $(\gamma^s_t)$. We will show

\begin{thm}\label{main}
  There are $s_1, \delta_1,$ and  $\tau_1$  depending on $\alpha<1/2,$ $D_1,$ $D_2,$  $D_3,$  $\psi,$ $ r_0,$ and $\varepsilon_0$,   so that  if 
  $$ t\leq\delta_1,\,   r^2\leq \tau_1t,\,\mbox{ and }\,s\leq s_1,$$ then
  $$\Theta^s_t(x_0,r)\leq 3/2+\varepsilon_0$$
  for every $x_0$ in $B_1(0)$.
\end{thm}

\begin{proof}
Throughout the proof it will be understood that, unless stated, all constants will depend only on $\alpha<1/2,$ $D_1,$ $D_2,$  $D_3,$  $\psi,$ $ r_0,$ and $\varepsilon_0.$ All the lemmas will be proven in section \ref{proofs}.

Set $$\tilde \gamma^s_t=\frac{1}{(2(s+t))^{1/2}}\gamma^s_t\, .$$  
We start by proving estimates that hold either for short-time or
  far from the origin. They will be simple consequences of Huisken's monotonicity formula.
  \begin{lemm}\label{density1} 

 [Far from origin estimate]
  There are
    $\delta_1$ and $K_0$  so that if $r^2\leq t \leq \delta_1$, then
    $$\Theta^s_t(x_0,r)\leq 3/2+\varepsilon_0$$
    for every $x_0$  with $1\geq |x_0|\geq K_0\sqrt{2t}$.\\
    
    [Short-time estimate]
    There are $s_1$ and $q_1$  such that if $s\leq s_1$, $r^2, t\leq q_1s$, then  
    \begin{equation}\label{total}
      \Theta^s_t(x,r)\leq 3/2+\varepsilon_0
    \end{equation}
    for every $x$ in $B_{1}(0).$
 \end{lemm}

 \begin{rmk}\label{remark}
 \begin{itemize}
 \item[1)]It follows from the second estimate that we need only to prove Theorem \ref{main} when  $t\geq q_1s$. 
 \item[2)]Setting
 $$\tilde \Theta^s_t(x,r)=\int_{\tilde \gamma^s_t}\Phi(x, r^2)d\H^1,$$
 and in virtue of
 $$\Theta^s_t(x_0,r)=\tilde \Theta^s_t\left(\frac{x_0}{(2(s+t))^{1/2}}, \frac{r}{(2(s+t))^{1/2}}\right),$$
 in order to prove Theorem \ref{main} it suffices to find $s_1, \delta_1,$ and $\tau_1$ such that for every $s\leq s_1$, $q_1s\leq t\leq \delta_1$, $r^2\leq \tau_1$, and $y_0$ with $|y _0|\leq (2(s+t))^{-1/2}$, we have
 $$\tilde \Theta^s_t(y_0,r)\leq 3/2+\varepsilon_0.$$
 \item[3)]
 Set $$\tau_1=q_1/(2(q_1+1)).$$
 The second estimate in the lemma implies that for $s\leq s_1$, $t\leq q_1 s$, and $r^2\leq \tau_1$ we have
 $$\tilde \Theta^s_t(y_0,r)\leq 3/2+\varepsilon_0$$ 
 for every $|y_0|\leq (2(s+t))^{-1/2}.$
 The first  estimate in Lemma \ref{density1} implies that for $r^2\leq \tau_1, s\leq s_1$ and $q_1s\leq t\leq \delta_1,$ 
$$\tilde \Theta^s_t(y_0,r)\leq 3/2+\varepsilon_0$$
  for every $y_0$ with $K_0\leq |y_0|\leq (2(s+t))^{-1/2}$.
  \end{itemize}
  \end{rmk}
  From now on, consider $K_0, q_1,  s_1,$ and $\delta_1$, given by Lemma \ref{density1} and set $\tau_1=q_1/(2(q_1+1))$.

  In the next two lemmas we control the asymptotic behavior  of $\tilde \gamma^s_t$. The proof will be a bit involving because it is important that $r_1$ does not depend on $\nu$.

  \begin{lemm}[Proximity to $P$]\label{asymptotic}
    There are $C_1$ and $r_1$ so that for every $\nu$ we can find $s_2$, and $\delta_2$ for which the following holds. If 
    $s\leq s_2$,  $t\leq \delta_2,$ and $r\leq 2$, then 
   $$\dist(y_0,P)\leq \nu+C_1\exp(-|y_0|^2/C_1)\quad\mbox{ if }\quad
   y_0\in \tilde \gamma^s_t\cap A\left(r_1,(s+t)^{-1/8}\right),$$
   and
   $$  \tilde\Theta^s_t(y_0,r)\leq 1+\varepsilon_0/2+\nu\quad\mbox{ if }\quad
   y_0\in  A\left(r_1,(s+t)^{-1/8}\right).$$ 
    \end{lemm}

Denote by $F^s_t$ the normal deformation
  $$ F_t^s:  \gamma^s \longrightarrow \R^2$$ such that $\gamma^s_t= F^s_t(\gamma^s)$ and set $\tilde F^s_t=(2(s+t))^{-1/2}F^s_t$ so that $\tilde \gamma^s_t= \tilde F^s_t(\gamma^s)$. Using the previous lemma with $\nu=\varepsilon_0/2$ we obtain, as we shall see in section \ref{proofs},

\begin{lemm}\label{graphical}
There are $r_2$, $\delta_3$, $s_3$,  and $L$,  such that if $t\leq \delta_3$ and  $s\leq s_3$ then
$$|\tilde F^s_0(x)-\tilde F^s_t(x)|\leq L\quad\mbox{  whenever }\quad\tilde F^s_0(x)\in A(r_2, (s+t)^{-1/8}/2).$$
\end{lemm}

 Consider $C_1$ and  $r_1$ given by Lemma \ref{asymptotic}, $r_2, \delta_3, s_3,$ and $L$ given by Lemma \ref{graphical}, and choose $\eta_1=\eta_1(\tau_1)$ given by Lemma \ref{junction}. We then set $r_3=\max\{r_0,r_1,r_2,1\}$. Apply Corollary \ref{unique} where we consider  $R=\sqrt{1+2q_1} K_0+r_3$,  $\varepsilon=\varepsilon(\psi, \alpha)$ to be the one given by Lemma \ref{shortime}, $r_1$ to be $r_3$, and $\eta_1,C_1,D_1,\alpha,$ and $P$ to be the constants already defined. Then, we get the existence of $R_1, \beta,$ and $\nu$ for which Corollary \ref{unique} holds.
 
 Consider now
  $s_2=s_2(\nu),$ 
  $\delta_2=\delta_2(\nu)$ given by Lemma
  \ref{asymptotic} and set $s_4=\min\{s_1,s_2, s_3\}$,  $\delta_4=\min\{\delta_1,\delta_2, \delta_3\}$.
  Finally decrease $s_4, \delta_4$ if necessary
  so that $$ (s_4+\delta_4)^{-1/8}\geq 2R_1.$$

The next lemma is essential to prove  Theorem \ref{main} and its content is that the proximity of $\tilde \gamma^s_t$ to a self-expander can be controlled in an integral sense. It is
the only place where we use the evolution equations derived in section
\ref{monotonicity}.  

Choose	$a>1$ such $(1+2q_1)/a >1$ and set $q=q_1/a.$

\begin{lemm}\label{integral}
  There are  $\delta_0$, and $s_0$   so that for every $$qs \leq T\leq
  \delta_0\quad\mbox{ and } \quad s\leq s_0,$$ we have
  $$\frac{1}{(a-1)T}\int^{aT}_{T}\int_{\tilde \gamma^s_t\cap B_{R_1}}|\vec
  k-x^{\bot}|^2d\H^1dt\leq \beta.$$
\end{lemm}

  Consider $\delta_0, s_0$ for which the lemma holds and set $s_5=\min\{s_0,s_4\}$, $\delta_5=\min\{\delta_0,\delta_4\}$. Decrease $s_5$ if necessary so that $q_1s_5\leq \delta_5.$

  Having all the constants properly defined, we can now finish the proof.
  Set
  $$T_0=\sup\{T\,\,|\,\; \tilde\Theta^s_{t}(x,r)\leq
  3/2+\varepsilon_0\quad\mbox{for all } x\in B_{K_0}(0),\; r^2\leq
  \tau_1, \;  t\leq T\}.$$
  It suffices  to show 
  that $T_0 \geq \delta_5$ for every  $s\leq s_5$. Remark \ref{remark} 1) implies that $T_0\geq q_1s$.
  
  Suppose that $T_0<\delta_5$ and set $T=T_0/a$. Lemma
  \ref{integral} implies the existence of $T \leq t_1\leq T_0$ so that
  $$\int_{\tilde \gamma^s_{t_1}\cap B_{R_1}}|\vec k-x^{\bot}|^2d\H^1\leq \beta.$$
  
  We now check that Corollary \ref{unique} can be applied with $\sigma$ being $\tilde \gamma^s_{t_1}$. Conditions i) and ii) are trivially satisfied. For every $0\leq t\leq t_1$ set 
  $$\hat \sigma_t=\tilde F^s_t(\gamma^s\cap B_{R_1+L}(0)).$$
  During the proof of Lemma \ref{graphical} we chose $r_2$ so that 
  $$\tilde \Theta^s_t(x,r)\leq 1+\varepsilon_0$$
  for every $r\leq 2$ and $x$ in $A(r_2,(s+t)^{-1/8})$. This implies that all the triple junctions of $\hat \sigma_t$ are inside $B_{r_3}(0)$. Lemma \ref{graphical} implies that the boundary points of $\hat \sigma_t$ lie outside $B_{R_1}(0)$, and so condition iii) b) is met. Condition iii) a)  holds because  Remark \ref{remark} 3) implies that, for every $x$ in $B_{R_1}(0)$, $r^2\leq \tau_1$, and $t\leq t_1$
    $$\tilde\Theta^s_t(x,r)\leq 3/2+\varepsilon_0,$$
   and so  Lemma \ref{junction} can be applied with $R=R_1$. Condition iii) c) is satisfied because of Lemma \ref{asymptotic}. Condition iii) d) is not immediately satisfied because $\hat \sigma_0$  coincides with part of $(2s)^{-1/2}\gamma^s$ instead of $\psi$. Nonetheless, using hypothesis H3) and picking $s_5$ smaller if necessary, one can extend the family $(\hat \sigma_t)_{0\leq t\leq t_1}$ so that condition iii) d) indeed holds.

  Therefore, we get from Corollary \ref{unique} that
    $\tilde \gamma^s_{t_1}$ is $\varepsilon$-close
    in  $C^{1,\alpha}(B_{R_1}(0))$ to $\psi$.
    Denote by $(\hat \gamma^s_l)_{l\geq 0}$ the solution to network
    flow with initial condition $\tilde \gamma^s_{t_1}$. A simple
    computation shows  that
    $$\hat \gamma^s_l=\sqrt{1+2l}\tilde \gamma^s_{t_1+l\lambda^2},$$
    where $\lambda^2=2(s+t_1).$ Applying Lemma \ref{shortime} we
    conclude that for every $l\leq q_1$
    $$\tilde\Theta^s_{t_1+l\lambda^2}(x,r)=\hat
    \Theta^s_{l}(\sqrt{1+2l}x,\sqrt{1+2l}r)\leq 3/2+\varepsilon_0$$
    provided 
    $$\sqrt{1+2l}|x|\leq R_1-1 \quad\mbox{and}\quad(1+2l)r^2\leq
    q_1.$$
    Hence, for all $t_1\leq t\leq t_1(1+2q_1),$
    $$\tilde\Theta^s_{t}(x,r)\leq 3/2+\varepsilon_0$$
    for every $x$ in $B_{K_0}(0)$ and $r^2\leq \tau_1$, which implies that
  $T_0\geq t_1(1+2q_1)$. This is a contradiction because
  $$t_1(1+2q_1)\geq T(1+2q_1)=T_0(1+2q_1)/a>T_0.$$
\end{proof}

\section{Omitted proofs from section \ref{sec:main theorem}}\label{proofs}

We prove the various lemmas used in the previous section.\\[-2ex]

\begin{proof}[Proof of Lemma \ref{density1}] We start by showing the existence of $K_0$ so that for every $y_0$ in $\R^2$ with $|y_0|\geq K_0$ and $\lambda >0$
  $$\int_{\lambda\left( \gamma^s\cap B_3(0)\right)}\Phi (y_0,1)d\H^1\leq 3/2+\varepsilon_0/2.$$ 
 
  We argue by contradiction and assume the existence of $y_i$ tending to infinity, $\lambda_i$, and $s_i$,  for which 
  \begin{equation}\label{bad}
  	\int_{\lambda_i\left( \gamma^{s_i}\cap B_3(0)\right)}\Phi (y_i,1)d\H^1\geq 3/2+\varepsilon_0/2.
	\end{equation}
	The first remark is that $(\lambda_i)_{i\in \N}$ has to be an unbounded sequence, because for some universal constant $C$ and for all $i$ sufficiently large
	$$\int_{\lambda_i\left( \gamma^{s_i}\cap B_3(0)\right)}\Phi (y_i,1)d\H^1\leq C\lambda_i\exp(-|y_i|^2/8+C\lambda_i^2)\H^1(\gamma^{s_i}\cap B_3(0)).$$
	The second remark is that  from hypothesis H3) and H4) it follows the existence of $D_4$ depending only on $\psi, r_0$ and $D_3$, so that on $\gamma^s\cap B_3(0)$
  $$|\vec k|\leq D_4\left(1+s^{-1/2}e^{-\frac{|x|^2}{4s}}\right)$$
  and thus, setting $\sigma_i=\lambda_i\gamma^{s_i}$ and $l_i=\lambda^2_is_i$, we have on $\sigma_i \cap B_{3\lambda_i}(0)$
  $$|\vec k|\leq D_4\left(\lambda^{-1}_i+{l_i}^{-1/2}e^{-\frac{|x|^2}{4l_i}}\right).$$
  Because $y_i$ is tending to infinity, it is easy to recognize that the curvature goes to zero uniformly on compact sets centered around $y_i$. As a result, $\sigma_i-y_i$ converges to either a line, or a union of half-lines. We first note that one only needs to consider the case $\lim_{i\rightarrow \infty} l_i = \infty$, otherwise H3) and $|y_i|\rightarrow \infty$ yield a contradiction to \eqref{bad}. Furthermore, all the triple junctions of $\sigma_i$ are inside a ball of radius proportional to $l_i^{1/2}$ and the shortest distance between them is also proportional to $l_i^{1/2}$. Hence, because $y_i$ is getting arbitrarily large, we see that $\sigma_i-y_i$ converges to a either  plane or a union of three half-lines. This contradicts inequality \eqref{bad}.

    Hypothesis H1) ensures us that we can choose $\delta_1$ so that for every $x_0$ in $B_1(0)$ and  $l\leq 2\delta_1$
    $$\int_{\gamma^s\setminus B_3(0)}\Phi(x_0,l)d\H^1\leq \varepsilon_0/2.$$
    The monotonicity formula implies that for $r^2, t\leq \delta_1$.
    \begin{multline*}
      \Theta^s_t(x_0,r)\leq \int_{\gamma^s}\Phi(x_0,{r^2+t})d\H^1\\
      =\int_{\gamma^s\setminus
        B_3(0)}\Phi(x_0,{r^2+t})d\H^1+\int_{\gamma^s\cap B_3(0)}\Phi(x_0,{r^2+t})d\H^1\\
      \leq \varepsilon_0/2+\int_{(r^2+t)^{-1/2}\left(\gamma^s\cap
        B_3(0)\right)}\Phi(x_0/\sqrt{r^2+t},1)d\H^1\\
    \leq 3/2+\varepsilon_0,
    \end{multline*}
    provided $|x_0|\geq K_0\sqrt{r^2+t}$. This proves the first statement.
    
    Pick
    $$\varepsilon=\varepsilon(\psi, \alpha),\quad
    q_1=q_1(\psi, \alpha)$$
    given by Lemma \ref{shortime} and apply this lemma with
    $$\sigma_t=(2s)^{-1/2}\gamma^s_{2st}\quad\mbox{and } R=K_0\sqrt q_1+1.$$
    Note that by hypothesis H3) we can choose $s_1$ so that, for every $s\leq s_1$, $\sigma_0$ is $\varepsilon$-close to $\psi$ in $C^{1,\alpha}(B_R(0))$ and $s_1q_1\leq \delta_1$.
    Scale invariance implies that  for every $s\leq
    s_1$, $r^2\leq t\leq q_1s$, and $x$ in $B_{\sqrt{2sq_1}K_0}(0),$
    $$\Theta^s_t(x,r)\leq 3/2+\varepsilon_0.$$
    This proves the second statement because  the ball  $B_{\sqrt{2sq_1}K_0}(0)$ contains $B_{\sqrt{2t}K_0}(0)$ if $ t\leq q_1s$.
\end{proof}

  \begin{proof}[Proof of Lemma \ref{asymptotic}]
    Set 
    $$l=t(2(s+t))^{-1}\mbox{ and }\sigma^s=(2(s+t))^{-1/2}\gamma^s.$$ 
    Note that $l \leq 1$. Moreover, for $s_2=s_2(r_0)$
    and $\delta_2=\delta_2(r_0)$ small we have that
    $$\sigma^s\cap A\left(r_0,3(s+t)^{-1/8}\right)$$ 
    is graphical over $P\cap A\left(r_0,3(s+t)^{-1/8}\right)$ and  if
    $v_s$ is a function arising from the graphical decomposition then
    \begin{multline*}
    	|v_s(x)|+ |x||{dv_s}/{dx}|+|x|^2\left|d^2v_s/ {dx^2}\right|\\
    \leq D_3 \left( 2(t+s)^{1/2}|x|^2 +\exp(-|x|^2/2)\right),
    \end{multline*}
    which means that, by choosing $s_2=s_2(D_3,r_0)$,  $\delta_2=\delta_2(D_3,r_0)$ small enough and choosing  $r_1=r_1(r_0,D_3)\geq\max\{r_0,1\}$ large enough, we can ensure that
    \begin{equation}\label{bn}
    		|v_s(x)|+ |x||{dv_s}/{dx}|\leq D_3 \left( 2(t+s)^{1/2}|x|^2 +\exp(-|x|^2/2) \right)\leq 1
	\end{equation}
   on $A\left(r_1,3(s+t)^{-1/8}\right)$.
    
      From now on pick
    $$y_0\in \tilde \gamma^s_t\cap A\left(3r_1+1,(s+t)^{-1/8}\right).$$
    From the monotonicity formula we have that 			
    $$1\leq \Theta^s_{0}(y_0(2(s+t))^{1/2}, \sqrt t)= \int_{\sigma^s}\Phi(y_0,l)d\H^1=A+B+C,$$
    where
    \begin{align*}
      A&= \int_{\sigma^s\setminus B_{3(s+t)^{-1/8}}} \Phi(y_0,l)d\H^1,\\
      B&=\int_{\sigma^s \cap B_{r_1}}\Phi(y_0,l)d\H^1,\\
      C&=\int_{\sigma^s \cap A\left(r_1, 3(s+t)^{-1/8}\right)}\Phi(y_0,l)d\H^1.
    \end{align*}

For every $x$ with $|x|\geq 3(s+t)^{-1/8}$, the bounds for $y_0$ imply  that
    $$|x-y_0|^2\geq |x|^2/3+|y_0|^2$$
    and so
    $$\Phi(y_0,l)\leq \sqrt 3\exp(-|y_0|^2/(4 l))\Phi(0,3l).$$
    Thus, we can find $C_1=C_1(D_1)$ for which
    \begin{multline*}
      A= \int_{\sigma^s\setminus B_{3(s+t)^{-1/8}}} \Phi(y_0,l)d\H^1\\
      \leq \sqrt 3\exp(-|y_0|^2/(4 l)) \int_{\sigma^s\setminus
        B_{3(s+t)^{-1/8}}}\Phi(0,3l)d\H^1\\
      \leq C_1\exp(-|y_0|^2/C_1).
    \end{multline*}
    
    To estimate the second term we proceed in the same way. For every
    $|x|\leq r_1$, the bounds for $y_0$ imply that 
    $$|x-y_0|^2\geq |x|^2+|y_0|^2/3	$$
    and so
    $$\Phi(y_0,l)\leq \exp(-|y_0|^2/(12 l))\Phi(0,l)\quad\mbox{for every }|x|\leq  r_1.$$
    Thus, we can find $C_1=C_1(D_1)$ for which
    $$B\leq \exp(-|y_0|^2/(12 l))\int_{\sigma^s\cap
      B_{r_1}}\Phi(0,l)d\H^1\leq C_1\exp(-|y_0|^2/C_1).$$

        Finally, we  estimate the third term.  Denote by $P_i$ the
    half-lines such that $P=\{P_i\}_{i=1}^N$, by $a_i$ the orthogonal projection  of $y_0$ on the line determined by $P_i$, and by $b_i$ the projection 
    fof $y_0$ onto the normal space of $P_i$ so that
      $$\dist(y_0,P)=\min\{|b_i|\}=|b_1|.$$
     Furthermore, denote by $\sigma^s_i$ the component of $$\sigma^s\cap A\left(r_1,3(s+t)^{-1/8}\right)$$ 
    which is graphical over $P_i\cap A\left(r_1,3(s+t)^{-1/8}\right)$ and by $v^i_s$ the correspondent graph function. It is easy to recognise that for $i=2,\ldots,N$, we have
    $|b_i|\geq c|y_0|$, where $c=c(P)$ is some positive constant. Relabel
    $r_1=r_1(r_0,D_3,P)$ such that \eqref{bn} holds with $c/\sqrt{2}$ instead
    of 1 on the right hand side. Hence 
    $$(v^i_s - b_i)^2\geq (c^2/2)|y_0|^2$$
     and so there is $C_1=C_1(D_1,P)$ such that
    \begin{multline*}
        	\int_{\sigma^s_i}\Phi(y_0,l)d\H^1\leq 2(4\pi l)^{-1/2}\int \exp\left(-\frac{(c^2/2)|y_0|^2+(x-a_i)^2}{4l}\right)dx\\
	\leq C_1\exp(-|y_0|^2/C_1).
    \end{multline*}
    
     As a result, we combine all these estimates and obtain that for some $C_1=C_1(D_1,P)$
    $$1\leq \int_{\sigma^s_1}\Phi(y_0,l)d\H^1+C_1\exp(-|y_0|^2/C_1).$$
    We relabel  $r_1$ one last time  and find $r_1=r_1(D_1, D_3, r_0, \psi, \alpha)$ so that 
    $$C_1\exp(-r_1^2/C_1)\leq 1/2.$$
	This combined with H1) implies that there exists a constant $c = c(D_1)$ such that
	$$ 1/4 \leq \int_{\sigma^s_1 \cap B_{c \sqrt{l}}(y_0)}\Phi(y_0,l)d\H^1 \leq 2 \sup_{|x-a_1| \leq c \sqrt{l}}\exp\left(-\frac{|v_s^1-b_1|^2}{4l}\right)\, .$$
	 By \eqref{bn}, the variation of $v^1_s(x)$ over the interval $|x-a_1| \leq c \sqrt{l}$ is $O(\sqrt{l})$ and thus there exists a constant $\tilde{c} = \tilde{c}(D_1,P)$ such that
	 $$\sup_{|x-a_1| \leq c \sqrt{l}} \frac{|v_s^1-b_1|^2}{4l} \leq \tilde{c}\, .$$
	   Therefore, using that $\lambda \leq (1 - \exp(-\lambda)) \exp(\tilde{c})$ for $\lambda \in [0,\tilde{c}]$, we see that we can find $C_1=C_1(D_1,P)$ for which
   \begin{equation*}\begin{split}
   \int_{|x-a_1| \leq c \sqrt{l}}\!\!\!\!\!\!\!\!\!\!\!\!\!\!\!\!\!\!\!\!\!\!& \qquad \frac{(v^1_s-b_1)^2}{4l}\frac{\exp(-(x-a_1)^2(4l)^{-1})}{(4\pi l)^{1/2}}dx\\
	&\leq  C \int_{|x-a_1| \leq c \sqrt{l}}\left(1 - \exp\left(-\frac{(v^1_s-b_1)^2}{4l}\right)\right) \frac{\exp(-(x-a_1)^2(4l)^{-1})}{(4\pi l)^{1/2}}dx\\
	   &\leq C\bigg(\int_{\{x \geq r_1\}}\sqrt{1+(dv^1_s/dx)^2}\,\frac{\exp(-(x-a_1)^2(4l)^{-1})}{(4\pi l)^{1/2}}dx\\
	&\qquad -\int_{\sigma^s_1\setminus B_{r_1}(0)}\Phi(y_0,l)d\H^1\bigg)\\
	   &\leq C\left(\int_{\{x \geq r_1\}}\sqrt{1+(dv^1_s/dx)^2}\frac{\exp(-(x-a_1)^2(4l)^{-1})}{(4\pi l)^{1/2}}dx-1\right)\\
	   &\quad +C_1\exp(-|y_0|^2/C_1)\\
	  &\leq \int_{\{x \geq r_1\}} C(dv^1_s/dx)^2\frac{\exp(-(x-a_1)^2(4l)^{-1})}{(4\pi l)^{1/2}}dx+C_1\exp(-|y_0|^2/C_1)
   \end{split}\end{equation*}

    and thus
   \begin{equation*}\begin{split}
   b_1^2 &\leq C_1\int_{\{x \geq r_1\}}\left( (v^1_s)^2+(dv^1_s/dx)^2\right)\frac{\exp(-(x-a_1)^2(4l)^{-1})}{(4\pi l)^{1/2}}dx\\
   &\quad +C_1\exp(-|y_0|^2/C_1).
   \end{split}\end{equation*}
    
    We observe that $|a_1|\geq c|y_0|$ for some constant $c=c(P)$  and that for every $0\leq l\leq 1$ we have 
    $$\frac{(x+a)^2}{2}+\frac{x^2}{4l}\geq \frac{a^2}{8}+\frac{x^2}{8l}.$$
    Thus, we obtain from \eqref{bn}  that, for some constant $C_1=C_1(D_1, D_3, P)$,
    \begin{multline*}
    \int\ (dv^1_s/dx)^2\frac{\exp(-(x-a_1)^2(4l)^{-1})}{(4\pi l)^{1/2}}dx\\
    \leq D_3 \int \left(\sqrt{s+t}|x|+\exp(-x^2/2)\right)\frac{\exp(-(x-a_1)^2(4l)^{-1})}{(4\pi l)^{1/2}}dx\\
    \leq C_1\sqrt{s+t}+D_3\int\exp(-x^2/2) \frac{\exp(-(x-a_1)^2(4l)^{-1})}{(4\pi l)^{1/2}}dx\\
    \leq C_1\sqrt{s+t}+D_3\int\exp(-(x+a_1)^2/2) \frac{\exp(-x^2(4l)^{-1})}{(4\pi l)^{1/2}}dx\\
    \leq C_1\sqrt{s+t}+D_3\exp(-a_1^2/8)\int \frac{\exp(-x^2(8l)^{-1})}{(4\pi l)^{1/2}}dx\\
    \leq C_1\sqrt{s+t}+C_1\exp{(-|y_0|^2/C_1)}.
    \end{multline*}
    The same type of estimate holds for the term 
    $$\int\ (v^1_s)^2 \frac{\exp(-(x-a_1)^2(4l)^{-1})}{(4\pi l)^{1/2}}dx$$
 and so we can choose $s_2$  and $\delta_2$  both depending on $D_1,D_3,\psi, r_0, \alpha,$ and $\nu$, such that  for every 
    $s\leq s_2$ and $t\leq \delta_2$ we have
    $$b_1=\dist(y_0,P)\leq \nu+C_1\exp(-|y_0|^2/C_1).$$

We now show that, by relabeling $r_1$, $s_2$, and $\delta_2$  if necessary, we also have
    $$\tilde\Theta^s_t(y_0,r)\leq 1+\varepsilon_0/2+\nu$$
    for every $r\leq 1$. The argument is almost  identical to what we have just done and so we will just point out the differences.
    We keep the same notation and assumptions. 
    Arguing in the very same way as we did before, we obtain the existence of $C_1=C_1(D_1, D_3,P)$ and  $r_1=r_1(D_1, D_3, r_0, \psi, \alpha)$ for which
    \begin{equation*}\begin{split}
    		&\tilde\Theta^s_t(y_0,r)\leq \int_{\sigma^s}\Phi(y_0,l+r^2)d\H^1\\
		&\leq  \int_{\sigma^s_1}\Phi(y_0,l+r^2)d\H^1+C_1\exp(-|y_0|^2/C_1)\\
		&\leq \int\sqrt{1+(dv^1_s/dx)^2}\frac{\exp\left(-(x-a_1)^2(4(l+r^2))^{-1}\right)}{(4\pi (l+r^2))^{1/2}}dx\\
		&\qquad+C_1\exp(-|y_0|^2/C_1)
		\end{split}
\end{equation*}
 \begin{equation*}\begin{split}
		&\leq 1+C_1  \int\ |dv^1_s/dx| \frac{\exp\left(-(x-a_1)^2(4(l+r^2))^{-1}\right)}{(4\pi (l+r^2))^{1/2}}dx\\
		&\qquad+C_1\exp(-|y_0|^2/C_1)\\
		&\leq  1+C_1\sqrt{s+t}
	 +C_1 \int \exp(-x^2/2) \frac{\exp\left(-(x-a_1)^2(4(l+r^2))^{-1}\right)}{(4\pi (l+r^2))^{1/2}}dx\\
		&\qquad +C_1\exp(-|y_0|^2/C_1).
\end{split}
\end{equation*}

Using the fact that $0\leq l\le 1$ we obtain
 \begin{equation*}\begin{split}
  &\tilde\Theta^s_t(y_0,r) \leq  1\\
  &\qquad +C_1\sqrt{s+t} +C_1 \int \exp(-(x+a_1)^2/2) \frac{\exp\left(-x^2(4(l+r^2))^{-1}\right)}{(4\pi (l+r^2))^{1/2}}dx\\
 &\qquad +C_1\exp(-|y_0|^2/C_1)\\
  &\leq1+C_1\sqrt{s+t}+C_1\exp(-a_1^2/40) \int\frac{\exp\left(-x^2(8(l+r^2))^{-1}\right)}{(4\pi (l+r^2))^{1/2}}dx\\
  & \qquad +C_1\exp(-|y_0|^2/C_1)\\
  &\leq 1+C_1\sqrt{s+t}+C_1\exp(-|y_0|^2/C_1)\\
  &\leq 1+\varepsilon_0/2+C_1\sqrt{s+t}.
\end{split}
\end{equation*}

Thus, like before, we can choose $s_2$, $\delta_2$ for which the result holds.
\end{proof}	

\begin{proof}[Proof of Lemma \ref{graphical}]
  From scale invariance and applying Lemma \ref{asymptotic} with
  $\nu=\varepsilon/2$, we can find $r_2\geq 1,$ $\delta_3$, and $s_3$
  such that if $t\leq \delta_3,$ and $s\leq s_3,$ then
  $$\Theta^s_t(x,r)\leq 1+\varepsilon_0$$
  whenever $r\leq2(2(s+t))^{1/2}$ and
  $$x\in A\left(r_2(2(s+t))^{1/2},(2(s+t))^{1/2}(s+t)^{-1/8}\right).$$
  
  Hence, from White's regularity Theorem \cite{White05}, we obtain the existence of a universal constant $C$ for which
  $$\left|\frac{d F_t^s}{dt}(p)\right|=|\vec k|\leq C t^{-1/2}$$
  whenever
  $$F_t^s(p)\in A\left(3r_2(2(s+t))^{1/2}/2,3(2(s+t))^{1/2}(s+t)^{-1/8}/4\right).$$
  Choosing a larger $r_2$ (depending on C and the previous $r_2$) and $\delta_3, s_3$ smaller if necessary, we obtain  after integrating the previous inequality that
  $$|F_t^s(p)-F_0^s(p)|\leq 2C\sqrt t$$
  whenever
  $$F_0^s(p)\in A\left(r_2(2(s+t))^{1/2},(2(s+t))^{1/2}(s+t)^{-1/8}/2\right).$$
  This finishes the proof.
\end{proof}

\begin{lemm}
  \label{thetabetaestimate}
There exists $\delta_5>0$, s.t. for $0<s,t<\delta_5$ it holds that
\begin{equation}
  |\vec{k}(x)|+|\theta^s_t(x)|+|\beta^s_t(x)| \leq D_4 \qquad \forall\, x \in
  \gamma^s_t\cap A(1/3,3)
\end{equation}

\end{lemm}
\begin{proof}
  By assumption H4) the estimate is true for $t=0$ and $s$
  sufficiently small. H4) furthermore implies that for $s$ sufficiently
  small, each component of $\gamma^s\cap A(1/8,8)$ is a graph, uniformly
  small $C^2$-norm over a half-line $P$. By Theorem \ref{thm:graph_local}
  this implies that there exists $\delta_5>0$ such that
  $\gamma_t^s\cap A(1/6,6)$ remains a graph with small gradient over
  $P$ for $0\leq t \leq \delta_5$.
 This already implies the first two estimates of the statement,
since $\theta^s_t$ is continuous in $t$.

 The estimates of Ecker and Huisken, \cite{EckerHuisken91}, for graphical
mean curvature flow then imply that
$$\gamma^s_t \cap A(1/5,5)$$
remains is a graph over $P$ with small $C^2$-norm for $0\leq t \leq
\delta_5$.
Let $(N^s_t)_{0\leq t \leq T}$ be a smooth parametrization of
the evolving network. Since $\gamma^s_t\cap A(1/5,5)$ is free of
triple junctions we can locally reparametrize $(N^s_t)_{0\leq t \leq
  \min(T,\delta_5)}$ such that 
$$\Big(\frac{\partial}{\partial t} N\Big)^T = X^T = 0$$
on $ A(1/4,4)$. Since $X = \vec{k}$, we have by the evolution equation
for $\beta_t$ that,
$$\Big|\frac{d}{dt} \beta_t\Big| \leq |\langle X, Jx\rangle| +2
|\theta_t| \leq C \, .$$
Decreasing $\delta_5$ further if necessary, this implies the second
part of statement. 
\end{proof}

\begin{proof}[Proof of Lemma \ref{integral}]
  Set $T_0=R^2(aT+s)+aT$. During this proof $C$ denotes a constant
  which is allowed to depend also on $a, R$, and  $q$ (but not $T$ and
  $s$). We have from the localized
  monotonicity 
  formula applied to $2(s+t)\theta^s+\beta^s$ (see Lemma \ref{thm:evol.local}) that
  \begin{multline*}
    \frac{1}{(a-1)T}\int^{aT}_{T}\int_{\tilde \gamma^s_t\cap
      B_R(0)}|\vec k-x^{\bot}|^2d\H^1dt \\
    =\frac{1}{(a-1)T}\int^{aT}_{T}(2(s+t))^{-3/2}\int_{\gamma^s_t\cap
      B_{R\sqrt{2(s+t)}}(0)}|2(s+t)\vec k-x^{\bot}|^2d\H^1dt\\
    =\frac{1}{(a-1)T}\int^{aT}_{T}(2(s+t))^{-3/2}\int_{\gamma^s_t\cap
      B_{R\sqrt{2(s+t)}}(0)}|\nabla(2(s+t)\theta^s+\beta^s)|^2d\H^1dt\\ \displaybreak[3]
    \leq
    \frac{C}{T}\int^{aT}_{T}(s+t)^{-3/2}(T_0-t)^{1/2}\int_{\gamma^s_t
    }\varphi |
    \nabla(2(s+t)\theta^s+\beta^s)|^2\rho(0,T_0-t)d\H^1dt\\ \displaybreak[3]
    \leq\frac{C}{T}(s+T)^{-3/2}(T_0-T)^{1/2}\int_{\gamma^s_T
    }\varphi(2(s+T)\theta^s+\beta^s)^2\rho(0,T_0-T)d\H^1\\
    +
    \frac{C}{T}(s+T)^{-3/2}(T_0-T)^{1/2}\int_T^{aT}\!\!\int_{\gamma^s_t\cap
      A(2,3)} (2(s+t)\theta^s_t+\beta^s_t)^2\rho(0,T_0-t)\, d\H^1dt\\ 
    \leq\frac{C}{T}(s+T)^{-3/2}(T_0-T)^{1/2}\int_{\gamma^s_0
    }\varphi(2s\theta^s+\beta^s)^2\rho(0,T_0)d\H^1\\ \displaybreak[3]
    +
    \frac{C}{T}(s+T)^{-3/2}(T_0-T)^{1/2}\int_0^{aT}\!\!\int_{\gamma^s_t\cap
      A(2,3)} (2(s+t)\theta^s_t+\beta^s_t)^2\rho(0,T_0-t)\, d\H^1dt\\
\leq\frac{C}{T(s+T)}\int_{\gamma^s_0
    }\varphi(2s\theta^s+\beta^s)^2\rho(0,T_0)d\H^1\\
    +
    \frac{C}{T(s+T)}\int_0^{aT}\int_{\gamma^s_t\cap
      A(2,3)} (2(s+t)\theta^s_t+\beta^s_t)^2\rho(0,T_0-t)\, d\H^1dt \\
    =A+B\, ,\\
  \end{multline*}
For the second term, using Lemma \ref{thetabetaestimate},  we have
  \begin{multline*}
    B\leq \frac{C ((s+aT)+1)}{T(s+T)}\int_0^{aT}\int_{\gamma^s_t\cap
      A(2,3)}|x|^4\rho(0,T_0-t)d\H^1dt\\
    \leq \frac{C ((s+aT)+1)}{T(s+T)}\int_0^{aT}(T_0-t)^2\int_{(T_0-t)^{-1/2}(\gamma^s_t\cap
      A(2,3))}\!\!\!\!|x|^4\rho(0,1)d\H^1dt\\
       \leq \frac{C ((s+aT)+1)T_0^3}{T(s+T)}\sup_{0<t<aT}
    \int_{(T_0-t)^{-1/2}(\gamma^s_t\cap
      A(2,3))}|x|^4e^{-|x|^2/4}d\H^1\\
       \leq CT_0\sup_{0<t<aT}
    \int_{(T_0-t)^{-1/2}(\gamma^s_t\cap
      A(2,3))}|x|^4e^{-|x|^2/4}d\H^1\ .
      \end{multline*}
  Note that $$T_0\leq R^2\delta_0(a+1/q)+a\delta_0$$
  and so we can choose $\delta_0$ small enough so that $B\leq \beta/2$.
  
  We now estimate the first term. Recall that if $\beta$ is primitive for the Liouville form on the network $\gamma$, then $\beta_l=l^{-2}\beta $ is primitive for the Liouville form on $l^{-1}\gamma$. Set
  $$\lambda=\frac{s}{T+s}\quad\mbox{ and }\quad l=\sqrt{2(T+s)}.$$
  Then
  \begin{align*}
    A & \leq\frac{C}{T(s+T)}\int_{\gamma^s\cap B_3}(2s\theta^s+\beta^s)^2\rho(0,T_0)d\H^1\\
    &=\frac{C(s+T)}{T}\int_{l^{-1}\left(\gamma^s\cap B_3\right)}
    (\lambda\theta^s+\beta_l^s)^2\rho(0,l^{-2}T_0)d\H^1\\
    &\leq C \int_{l^{-1}\left(\gamma^s\cap
        B_3\right)}(\lambda\theta^s+\beta_l^s)^2\rho(0,l^{-2}T_0)d\H^1,
  \end{align*}
  where the last equality follows because $T\geq qs$. 
  Consider
  $$F(T,s)=\int_{l^{-1}\left(\gamma^s\cap
      B_3\right)}(\lambda\theta^s+\beta_l^s)^2\rho(0,l^{-2}T_0)d\H^1,$$
      where we remark the existence of a constant $C$ (independent of $T$ and
  $s$) such that $$C^{-1}\leq l^{-2}T_0\leq C.$$ 
  Given any $\beta_1$ small it is enough to show the existence of $s_0$ and $\delta_0$ so that if
  $qs \leq T\leq \delta_0$ { and } $s\leq s_0$ then
  $$F(T,s)\leq \beta_1.$$

  We now argue by contradiction and  assume the existence of $s_i$ and
  $T_i$  (with $qs_i\leq T_i$) converging to zero for which
  $F(T_i,s_i)\geq \beta_1$. We also assume that $l_i^{-2}T_0$ converges to $T_1$.
  
  Suppose first that $\lambda_i$ (as defined above) has a subsequence
  converging to a positive number $\lambda.$ In that case
  $$l_i^{-1}\gamma^{s_i}= \lambda^{1/2}_i \tilde \gamma^{s_i}$$
  converges in $C^{1,\alpha}_{\l}$ to $\lambda^{1/2}\psi$. Hypothesis H3) implies that
  $$\lim_i F(T_i,s_i)\leq \lim_i \lambda_i^2 \int_{\tilde \gamma^{s_i}}(\theta^s+\tilde \beta^{s_i})^2\rho(0,l_i^{-2}\lambda^{-1}_{i}T_0)d\H^1=0.$$
  Suppose now that  $\lambda_i$  has a subsequence converging to zero.
  It follows at once that
  $$\lim_i \int_{l_i^{-1}\left(\gamma^{s_i}\cap B_{r_0\sqrt
        s_i}\right)}(\lambda_i\theta^{s_i}+\beta_{l_i}^{s_i})^2\rho(0,l_i^{-2}T_0)d\H^1=0.$$
  Note that by hypothesis H4)
  $$l_i^{-1}\gamma^{s_i}\cap A(r_0(\lambda_i/2)^{1/2}, 3 l_i^{-1})$$
  is graphical over $P$ and  if $v_i$  is the function arising from
  the graphical decomposition of $l_i^{-1}\gamma^{s_i}$ then
  $$|v_i(x)|+ |x||{dv_i}/{dx}|+|x|^2\left|d^2v_i/ {dx^2}\right| \leq D_3 \left(l_i |x|^2+(\lambda_i)^{1/2}e^{-\frac{|x|^2}{2\lambda_i}}\right).$$
  Therefore, we have that
  \begin{equation}\label{gradient}
  |\nabla \beta_{l_i}^{s_i}|=|x^{\bot}|=\frac{|xv'_i-v_i|}{\sqrt{1+(v'_i)^2}}\leq D_3 \left(l_i |x|^2+(\lambda_i)^{1/2}\right).
  \end{equation}
  We now argue that for any connected component of 
  $$l_i^{-1}\gamma^{s_i}\cap A(r_0(\lambda_i/2)^{1/2}, 3 l_i^{-1})$$ there is  $x_i$ converging to zero for which
 $b_i=\beta_{l_i}^{s_i}(x_i)$ also converges to zero. From hypothesis  H3)  we see that any connected components of $\tilde \psi^{s_i}\cap A(2r_0, 3r_0)$ contains  $y_i$ such that
 $$\lim_{i}(\theta^{s_i}(y_i)+\tilde\beta^{s_i}(y_i))=0.$$
 Setting $x_i=\lambda_i y_i$, it is simple to see that $b_i=\lambda_i\tilde \beta^{s_i}(y_i)$ tends to zero.
 
 Therefore, we can use gradient estimate \eqref{gradient} and the
 graphical decomposition  to conclude  the existence of a  constant $C$
 independent of $i$ such that
 $$|\beta_{l_i}^{s_i}(x)|\leq C\left(l_i|x|^3+(\lambda_i)^{1/2}|x|\right)+b_i\quad\mbox{ on } \quad A(r_0(\lambda_i/2)^{1/2}, 3 l_i^{-1}).$$
 Hence,
 \begin{multline*}
   \lim_i F(T_i,s_i)  =\lim_i \int_{l_i^{-1}\gamma^{s_i}\cap
     A(r_0(\lambda_i/2)^{1/2}, 3
     l_i^{-1})}(\lambda_i\theta^{s_i}+\beta_{l_i}^{s_i})^2\rho(0,l_i^{-2}T_0)d\H^1\\
    = \lim_i \int_{l_i^{-1}\gamma^{s_i}\cap A(r_0(\lambda_i/2)^{1/2},
     3 l_i^{-1})}(\beta_{l_i}^{s_i})^2\rho(0,l_i^{-2}T_0)d\H^1\\
   \leq \lim_i C(l_i^2+\lambda_i+b_i^2) \int_{l_i^{-1}\gamma^{s_i}} (|x|^6+|x|^2+1)\rho(0,l_i^{-2}T_0)d\H^1=0.
 \end{multline*}
 This is a contradiction.
 \end{proof}

\section{Short-time existence}\label{sec:short time ex}
We will show in this section that one can glue a scaled self-expander
at scale $s$
into the initial network around a non-regular multiple point to obtain
a sequence of regular networks $\gamma^s$ which satisfy the hypotheses
H1)-H4) in section \ref{sec:main theorem}. We will show that combining Theorem
\ref{main} and \ref{thm:locreg.2} then proves short time existence
of the network flow for non-regular initial networks, Theorem
\ref{thm:short time ex}.

We will first discuss the question of short-time existence for regular
networks with unbounded branches. By the definition of a regular
network there exists an $R_0>0$ such that outside of $B_{R_0}$ the
initial network $\gamma_0$ consists of a finite number of non-compact
 branches $\gamma_0^i, \ i=1,\ldots,n$ which can be written as graphs
 over corresponding half-lines $P_i$. Since the curvature of
 $\gamma_0$ is bounded and the $\gamma_0^i$ approach the half-lines
 $P_i$ at infinity, we assume that each $\gamma_0^i$ can be written as
 a normal graph over $P_i$ with small $C^1$-norm. We define the points
 $q^{i}_k:=\gamma_0^i\cap \partial B_k$ for $k\geq k_0>R_0;\, k,k_0 \in \N$. By the
 results in \cite{MNTNetworks, Mantegazza04} there exists a maximal
 solution $(\gamma_{k,t})_{0\leq t<T_k}$ of the network flow, starting at
 $\gamma_0\cap B_k$  with fixed endpoints $q^i_k$. Using Proposition \ref{thm:locreg} in $B_{k_0}$ and Theorem
 \ref{thm:graph_local} to control the boundary points as well as
 estimates of Ecker and Huisken for graphical mean curvature flow to
 control the parts outside of $B_{k_0}$ to see that there is $T>0$
 such that $T_k\geq T$ for all $k\geq k_0$, together with uniform
 estimates on the curvature. We thus take a limit $k\rightarrow
 \infty$ to obtain a solution to the network flow, starting at $\gamma_0$. 

Now let $\gamma$ be a non-regular initial network with bounded
curvature. For simplicity let us assume that $\gamma$ has only one
non-regular multiple point at the origin.

If the multiple point consists only of two branches meeting at an
angle different than $\pi$, then smoothing the cone point and using estimates for
graphical mean curvature flow, see
for example the proof of Lemma \ref{thetabetaestimate}, one easily
constructs a solution starting at $\gamma$ as claimed in Theorem
\ref{thm:short time ex}.

So we can assume that at the origin at least three branches meet, and let
$T_j, \ j=1, \ldots, n,$ be the exterior unit normals. We denote with 
$$P_j = \{-t T_j\, | \ t\geq 0\}$$
be the corresponding half-lines. Since $\gamma$ has bounded curvature,
we can assume, by scaling $\gamma$ if necessary, that $\gamma \cap
B_5$ consists of $n$ branches $\gamma_j$ corresponding to the
half-lines $P_j$; and if $\theta_j$ is the angle that $P_j$ makes with the $x$-axis, there
is a function $u^j$ such that $\gamma_j$ can be parametrized as
$$\gamma_j = \{ x e^{i\theta_j}+u^j(x) e^{i(\theta_j + \pi/2)}\, |\, 0\leq
x \leq 5\}\, .$$
Note that the assumption that $\gamma$ has bounded curvature implies
\begin{equation}
  \label{eq:shorttime.1}
  |u^j(x)|\leq C x^2\qquad \text{and}\qquad \Big|\frac{d}{dx}
  u^j(x)\Big| \leq C x\, .
\end{equation}
 In \cite{SchnuererSchulze_2007} it was shown that for $n=3$ there
 exists a unique tree-like self-expander $\psi$ asymptotic to $P:=\cup_{j=1}^nP_j$. In
 the case $n>3$ the existence of tree-like, connected self-expanders was shown by Mazzeo-Saez
 \cite{MazzeoSaez07}. Note that Lemma \ref{decay} gives the
 asymptotics of $\psi$ outside a large ball $B_{r_0}$. 

We now aim to glue $\psi_s:=\sqrt{2s}\cdot\psi$ into $\gamma$ to get a
family satisfying the conditions H1)-H4). Let $v^j_s$ be the graph
function corresponding to the branch $\psi^j_s$ on $A(r_0 \sqrt{2s},
4)$. By Lemma \ref{decay} we have the estimate
\begin{equation}
  \label{eq:shorttime.2}\begin{split}
  &|v^j_s| \leq (2s)^{1/2} C e^{-x^2/4s}\ ,\ |(v^j_s)'| \leq x^{-1}
  (2s)^{1/2} C e^{-x^2/4s}\ ,\\ &|(v^j_s)''| \leq
  (2s)^{-1/2} C  e^{-x^2/4s}\ .
\end{split}
\end{equation}
 Let $\varphi:\R^+ \rightarrow [0,1]$ be a cut-off function
 s.t. $\varphi =1$ on $[0,1]$ and $\varphi = 0$ on $[2,\infty)$. We
 define $\gamma^s$ via the graph function $u^j_s$ in the gluing region $A(r_0 \sqrt{2s},
4)$ by
$$u^j_s:= \varphi(s^{-1/4}x)v^j_s(x) + (1-\varphi(s^{-1/4}x))u^j(x)\,
.$$
It can easily be checked that $\gamma^s$ satisfies the assumptions
H1)--H3). From \eqref{eq:shorttime.1} and \eqref{eq:shorttime.2} we see
that
$$|u^j_s| \leq C\Big(x^2+ (2s)^{1/2} e^{-x^2/4s}\Big)\ .$$
Furthermore 
\begin{equation*}\begin{split}
(u^j_s)' &=
s^{-1/4}\varphi'(s^{-1/4}x)v^j_s(x)+\varphi(s^{-1/4}x)(v^j_s)'(x)\\
&\ \ \ -s^{-1/4}\varphi'(s^{-1/4}x)u^j(x)+(1-\varphi(s^{-1/4}x))(u^j)'(x)\
. 
\end{split}\end{equation*}
We have $x^{-1} \leq s^{-1/4}\leq 2x^{-1}$ on $\{\varphi'(s^{-1/4}x)
\neq 0\}$ and so we can estimate
\begin{equation*}\begin{split}
x|(u^j_s)'| & \leq C(|v^j_s(x)|+ x|(v^j_s)'(x)|+|u^j(x)|+x
|(u^j)'(x)|\\
& \leq C\Big(x^2+ (2s)^{1/2} e^{-x^2/4s}\Big)\, .
\end{split}\end{equation*}
The estimate for $(u^j_s)''$ follows similarly, which shows that also
H4) is satisfied.

\begin{proof}[Proof of Theorem \ref{thm:short time ex}]
As discussed at the beginning of this section
there exits a smooth solution to the network flow $(\gamma^s_t)_{0\leq
  t \leq T_s}$ for some $T_s>0$. We now aim to show that there exists a $T_0>0$ such
that $T_s\geq T_0$ for all $s>0$ and that there are time interior
estimates on $k$ and all its higher derivatives for all positive
times, independent of $s$. 

Using Theorem \ref{thm:graph_local} and interior estimates for higher
derivatives of the curvature we see that we can pick a smooth family
of points $P_j(t,s) \in
\gamma_s^j\cap A(1/3,1/2)$ such that assumptions \eqref{eq:locreg.7}
and \eqref{eq:locreg.8} are satisfied, with constants independent of
$s$ for $0\leq t< \min\{T_s,\delta\}$, where $\delta>0$ does not depend on $s$. Then Proposition \ref{thm:locreg} gives
estimates on the curvature and its derivatives,
independent of $s$ on $\R^2\setminus B_{1/2}\times (0,\min\{T_s,\delta\})$. 

To get the desired estimates on $B_{1/2}$ we aim to apply Theorem
\ref{main} and Theorem \ref{thm:locreg.2}. Fix $\ve_0>0$ such
that $3/2+\ve_0 < \Theta_{S^1}$, and let $s_1,\delta_1,\tau_1$ be
determined by Theorem \ref{main}. 

Pick  $0<t_0<\min\{T_s,\delta_1,\delta \}$ and $x_0 \in B_{1/2}$. Let
$\rho:=(t_0/2)^{1/2}$. Note that $B_\rho(x_0)\subset B_1$. Theorem
\ref{main} then implies that the Gaussian density ratios
$$
\Theta(x,t,r)\leq 3/2 + \ve_0$$
for all $(x,t) \in B_\rho(x_0)\times (t_0-\rho^2,t_0)$ and $r \leq
\sqrt{\tau_1} \rho$. Thus by Theorem \ref{thm:locreg.2} with $\sigma
=1$, there exists $C$, depending only on $\epsilon_0, \tau_1$ such
that
$$ |k|(x_0,t_0) \leq \frac{C}{t_0^{1/2}}\ ,$$
together with the corresponding estimates on all higher
derivatives. By Remark \ref{rem:locreg.2} there is a $\kappa>0$,
depending only on $\epsilon_0, \tau_1$ such that the length of the
shortest segment is bounded from below by $\kappa \cdot t_0^{1/2}$.

Together with the estimate on $(\R^2\setminus B_{1/2})\times (0,\min\{T_s,\delta\})$
this implies that $T_s\geq T_0:=\min\{\delta, \delta_1\}$. By the estimates
on the curvature, which are independent of $s$ we can take a
subsequential limit of the flows $(\gamma^s_t)_{0<t<\bar{T}}$ as $s
\rightarrow 0$ to obtain a limiting flow $(\gamma_t)_{0<t<\bar{T}}$
starting at the non-regular network $\gamma$.

Note that by Theorem \ref{thm:graph_local} and the interior estimates of
Ecker/Huisken,  away from any triple and multiple point, the flow
$(\gamma_t)_{0<t<\bar{T}}$ attains the initial network $\gamma$ in
$C^\infty$. Furthermore by the above estimate in $B_1$
and Proposition \ref{thm:locreg} we have
$$ |k| \leq \frac{C}{t^{1/2}}\ .$$
The estimate on the length of the shortest segment passes to the limit
as well.
\end{proof}
\section{Local Regularity}\label{sec:loc reg}

\noindent In this section we will prove some local regularity results for the
network flow.
\subsection{Integral estimates}
We will need to localize the 
integral estimates in the work of Mantegazza, Novaga and Tortorelli
\cite{MNTNetworks}. In the following we will
outline what modifications of the original proofs are needed to obtain
the estimates in the local case. The setup is as follows.

Let $(\gamma_t)_{t\in [0,T)}$ be a regular, smooth solution of the network flow on
$\R^2$. Let $N:\gamma_0\times [0,T) \rightarrow \R^2$ be a smooth and regular
parametrization of the flow. We denote the tangential
component of the deformation vector by
\begin{equation}\label{eq:locreg.1}
\vec{\lambda} = X^T\, . 
\end{equation}
As defined before we denote with $T_j$ the exterior unit tangent
vector induced by each $\sigma_j$ at each triple point. We then define
\begin{equation}\label{eq:locreg.2}
 k_i = \langle \vec{k}_i, JT_i \rangle \ \ \ \text{and} \ \ \
\lambda_i = \langle \vec{\lambda}, T_i \rangle\, .
\end{equation}
The balancing condition at each triple point then implies
\begin{equation}
  \label{eq:locreg.3}
  k_1+k_2+k_3 = 0 \ \ \text{and}\ \ \lambda_1+\lambda_2+\lambda_3 =
  0\, .
\end{equation}
We would like to point out that our setup differs to the one in
\cite{MNTNetworks} in that we do not want to prescribe the tangential
component $\vec{\lambda}$ of the deformation vector. If one aims to
prove a short-time existence result, one has to specify the tangential
velocity. Nevertheless, and this is important in the following
discussion, the integral estimates on the curvature and higher
derivatives of the curvature of the evolving network do
not depend on the choice of tangential velocity. Another
point is that the calculations in \cite{MNTNetworks} are done only for
a network consisting of three curves, meeting at one common triple
point, and with three fixed endpoints. As already mentioned there, see
Remark 3.24 in \cite{MNTNetworks}, these calculations generalize
without any changes to networks with more than one triple point, but
with fixed endpoints.
In the following we will explain how to generalize these estimates to 
networks with arbitrary tangential speed, more than one triple point
and any number of moving endpoints.

We assume that along each segment $\sigma_i$ we have fixed an
orientation and thus the unit tangent vector field $\tau$ along is
well defined. Note that at each endpoint $p$ of $\sigma_i$ we have
$$\tau(p)=\pm T_i\, ,$$
depending on the chosen orientation. We fix the unit normal vector
field $\nu$ along $N$ by requiring that
$$ J\tau = \nu\, .$$
This convention implies that the curvature $k$ of $N$ is given by
$$k= \langle \vec{k},\nu\rangle = \langle \partial_s \tau,\nu\rangle =
-\langle \partial_s \nu,\tau\rangle\, ,$$
where $s$ is the arc-length parameter along $\sigma_i(t)$. Similarly
we define
$$\lambda = \langle X, \tau \rangle =\langle \vec{\lambda}, \tau \rangle\, .$$
 Note that
this implies again that at an endpoint $p$ of $\sigma_i$ it holds that
$$k(p)=\pm k_i \ \ \ \text{and}\ \ \ \lambda(p)= \pm \lambda_i\, .$$
It then can be easily checked that the evolution equations for $\tau,
\nu$ and $k$ do not depend on the choice of the tangential speed
$\vec{\lambda}$ and are given by
\begin{align}
  \label{eq:locreg.4} \partial_t\tau &= (\langle \nabla k,\tau\rangle
  + k \lambda)\nu\\[0.5ex]
\label{eq:locreg.5}\partial_t \nu &= -(\langle \nabla k,\tau\rangle+k\lambda)\tau\\[0.5ex]
\label{eq:locreg.6}\partial_t k &= \Delta k+ \langle \nabla
k,\vec{\lambda}\rangle  + k^3\, ,
\end{align}
see $(2.4)$, $(2.5)$ and $(2.6)$ in \cite{MNTNetworks}. Furthermore
the estimates and relations between the curvatures $k_i$ and the
tangential speeds $\lambda_i$ at a
triple point, following $(2.6)$ until $(2.10)$ in \cite{MNTNetworks},
remain valid. As well the evolution equation for higher derivatives of
the curvature and the relations between time and spacial derivatives given
in Lemma 3.7 and the calculus rules in Remark 3.9 in
\cite{MNTNetworks} are not affected. 

This ensures that all the calculations for integrals of the curvature
and its derivatives are identical up to contributions from the
boundary points. To control the influence of the boundary points we
make the following assumption.\\[2ex]
{\bf Assumption:} We assume that the evolving network
$(\gamma_t)_{t\in [0,T)} $ has boundary points $Q_l(t)$, where $l =
1,\ldots, N$. We assume that these boundary points are all disjoint
 and at each of this points it holds that
\begin{equation}
  \label{eq:locreg.7}
  X^T|_{(P(t),t)} = \vec{\lambda}(Q_l(t),t) = 0\, .
\end{equation}
for all $t\in [0,T)$. Furthermore we assume that there are positive constants
$C_j$ such that
\begin{equation}
  \label{eq:locreg.8}
   \sup_{l\in \{1,\ldots, N\}}|\nabla^jk|\big|_{(Q_l(t),t)} \leq C_j 
\end{equation}
for all $j=0,1,\ldots, j_0$, where $j_0 \in \N$, $t\in [0,T)$.\\[2ex]
With this assumption the additional terms in the evolution of the
integral of the square of $\partial_s^jk$ can be controlled. To
demonstrate this, and for the reader's convenience, we will do this calculation explicitly, compare with
$(3.4)$ in \cite{MNTNetworks}. 
\begin{equation}
  \label{eq:locreg.9}
  \begin{split}
    \frac{d}{dt}\int_{\gamma_t} |&\nabla^jk|^2ds = 2
    \int_{\gamma_t} \nabla^jk\, \partial_t \nabla^jk\, ds\\
    &\ \ \  +
    \int_{\gamma_t} |\nabla^jk|^2(\text{div}(\vec{\lambda}) - k^2)\, ds\\
    &= 2 \int_{\gamma_t} \nabla^jk\,  \Delta\nabla^{j}k +
      \nabla_{\vec{\lambda}}\nabla^{j}k\, \nabla^jk\, ds\\
&\ \ \  +   \int_{\gamma_t} \mathfrak{p}_{j+3}(\nabla^jk)\,\nabla^jk
\, ds
+\int_{\gamma_t} |\nabla^jk|^2(\text{div}(\vec{\lambda}) - k^2)\, ds
\\ 
&= -2\int_{\gamma_t} |\nabla^{j+1}k|^2ds +
\int_{\gamma_t} \text{div}(\vec{\lambda}|\nabla^jk|^2)\, ds\\
&\ \ \ + \int_{\gamma_t} \mathfrak{p}_{2j+4}(\nabla^jk)\, ds +
\sum_\text{3-points}\sum_{i=1}^3\langle T_i,\nabla|\nabla^jk|^2\rangle\bigg|_\text{3-point}\\
&\ \ \ +\sum_{l=1}^N\langle T_l,\nabla|\nabla^jk|^2\rangle\bigg|_\text{$Q_l$}\\
&= -2\int_{\gamma_t} |\nabla^{j+1}k|^2ds + \int_{\gamma_t}
\mathfrak{p}_{2j+4}(\nabla^jk)\, ds\\
&\ \ \  
\sum_\text{3-points}\sum_{i=1}^3\langle
T_i,\nabla|\nabla^jk|^2\rangle + \lambda_i| 
\nabla^jk|^2\bigg|_\text{3-point}\\
&\ \ \   +\sum_{l=1}^N\langle
T_l,\nabla|\nabla^jk|^2\rangle\bigg|_{Q_l}\, .
  \end{split}
\end{equation}
In the special case $j=0$ one gets
\begin{equation*}\begin{split}
 \frac{d}{dt}\int_{\gamma_t} &k^2ds = -2\int_{\gamma_t} |\nabla k|^2ds
 + \int_{\gamma_t}k^4\, ds \\
&\ \ \ \ + 
\sum_\text{3-points}\sum_{i=1}^3\langle
T_i,\nabla(k^2)\rangle + \lambda_ik^2\bigg|_\text{3-point}
+\sum_{l=1}^N\langle T_l,\nabla(k^2)\rangle\bigg|_\text{$Q_l$}\, .
\end{split}\end{equation*}
The relations at the triple points, see $(2.10)$ in \cite{MNTNetworks},
imply that at each triple point
\[\sum_{i=1}^3 2\langle
T_i,\nabla(k^2)\rangle + \lambda_ik^2\bigg|_\text{3-point} = 0\ .\]
Thus the order of differentiation at the triple point can lowered by
one order, and one gets
\begin{equation}\label{eq:locreg.10}
\begin{split}
 \frac{d}{dt}\int_{\gamma_t} k^2ds &= -2\int_{\gamma_t} |\nabla k|^2ds
 + \int_{\gamma_t}k^4\, ds \\
&\ \ \ \ - 
\sum_\text{3-points}\sum_{i=1}^3\lambda_ik^2\bigg|_\text{3-point}
+\sum_{l=1}^N\langle T_l,\nabla(k^2)\rangle\bigg|_\text{$Q_l$}\, .
\end{split}\end{equation}
Following verbatim the computations in \cite{MNTNetworks} on can use
interpolation inequalities for $L^p$-norms of $k$ and higher
derivatives of $k$ to absorb the term
$$\int_{\gamma_t}k^4\, ds$$
and the boundary terms at the triple-points. Note that the
contributions at the boundary points $Q_l$ are bounded by
$NC_0C_1$. This leads to the estimate, compare $(3.10)$ in
\cite{MNTNetworks},
\begin{equation}
  \label{eq:locreg.11}
   \frac{d}{dt}\int_{\gamma_t} k^2ds \leq C \bigg(1 + \int_{\gamma_t}k^2\,
   ds\bigg)^3\, ,
\end{equation}
where $C$ depends only on a  bound for the inverses of the
lengths of the segments the evolving network and $NC_0C_1$. This
inequality implies that the $L^2$-norm of $k$ cannot grow to quickly.
It can be furthermore shown that an estimate for the $L^2$-norm of
every even derivative $\nabla^jk$ is true, which depends only on the
$L^2$-norm of $k$, a  bound for the inverses of the
lengths of the segments of the evolving network and
$NC_jC_{j+1}$. Compare here the proof of Proposition 3.13 in
\cite{MNTNetworks}. 

A bound for the inverses of the
lengths of the segments $l(\sigma_i)$ of the evolving network,
depending on the initial network and $\int k^2$ is also true. Note that since at the endpoints $Q_l$ we have
$\lambda_l=0$ there is no extra contribution there. As in the proof of
Proposition 3.15 in \cite{MNTNetworks} one obtains
\begin{equation}\label{odek2linv}
\frac{d}{dt}\bigg(1 + \int_{\gamma_t}k^2 + \sum_i \frac{1}{l(\sigma_i)}\bigg)
\leq C\bigg(1 + \int_{\gamma_t}k^2 + \sum_i
\frac{1}{l(\sigma_i)}\bigg)^3 
\end{equation}
where $C$ depends only on $NC_0C_1$. Thus also the length of the
shortest segment remains bounded from below for a short time. Thus
there exists a $T_0>0$, depending only $L^2$-norm of the curvature of
$\gamma_0$, the inverses of the lengths of the segments of $\gamma_0$,
and $N, C_1, C_2$ such that on $[0,T_0]$ the $L^2$-norm of $k$ and the
inverse of the length of the shortest segment remains uniformly bounded.

To obtain estimates for higher derivatives of $k$ which are interior
in time, Mantegazza, Novaga and Tortorelli look, for $j$ even, at the evolution of
integrals of the form, 
\[\int_{\gamma_t} k^2 + \frac{t}{2!} |\nabla k|^2 + \cdots
+\frac{t^j}{j!}|\nabla^jk|^2\, ds\ . \]
By \eqref{eq:locreg.9} we get for the time derivative of such a quantity
in our case only the additional boundary term
\[\sum_{l=1}^N\sum_{i=0}^j\frac{t^i}{i!}\langle
T_l,\nabla|\nabla^ik|^2\rangle\bigg|_\text{$Q_l$}\ ,\]
which by our assumption \eqref{eq:locreg.8} is bounded on finite time
intervals. So arguing as in \cite{MNTNetworks} we obtain, compare
p.\,273 there, that on $[0, T_0]$
\[\int_{\gamma_t} k^2 + \frac{t}{2!}|\nabla k|^2 + \cdots
+\frac{t^j}{j!}|\nabla^jk|^2\, ds \leq \tilde{C
}_j\ .\]
Here the constants $\tilde{C}_j$ depend only on the $L^2$-norm of the
curvature of $\gamma_0$, the inverses of the lengths of the segments
of $\gamma_0$ and the constants $C_1, \ldots , C_{j+1}$. Using
interpolation inequalities, see Remark 3.12 in \cite{MNTNetworks}, we
can thus state the following Proposition.

\begin{prop}
  \label{thm:locreg}
Let $(\gamma_t)_{t\in [0,T)}$ be a smooth solution to the network flow,
with $N$ endpoints, satisfying the assumptions \eqref{eq:locreg.7} and
\eqref{eq:locreg.8}. Then there exists $T_0>0$, depending only on the
$L^2$-norm of the curvature of
$\gamma_0$, the inverses of the lengths of the segments of $\gamma_0$,
and $N, C_1, C_2$ such that for all $0<t<\min\{T, T_0\}$ it holds fo
all $j>0$ that
\[ |\nabla^j k | \leq \hat{C}_j \cdot t^{-\frac{j}{2}-\frac{1}{4}}\, ,\]
where $\hat{C}_j$ depends only on the
$L^2$-norm of the curvature of
$\gamma_0$, the inverses of the lengths of the segments of $\gamma_0$,
$N$ and the constants $C_1, \ldots , C_{j+1}$.
\end{prop}

\subsection{Generalized self-similarly shrinking networks}
In the following we define a {\it degenerate regular
  network}. It can be seen as a $C^1$-limit of regular networks, where
it is allowed that the lengths of some segments go to zero.

\begin{defi}[Degenerate regular network] We consider a connected graph $G$ consisting of a finite number of edges $e_i,\
1\leq i \leq N_1$ and vertices $v_j, 1\leq j \leq N_2$. We assume that the
edges are either homeomorphic to the interval $[0,1]$, with two boundary
points, or homeomorphic to $[0,\infty)$, with one boundary point. We assume that
at the vertices always three such boundary points meet. We furthermore
assume that there exists a continuous map  $\Psi: G \rightarrow T\R^2, x \mapsto
(\Psi(x), \Psi^\prime(x))$ such that if $e_i$ is homeomorphic to a
finite interval, then either
\begin{itemize}
\item[i)] $\Psi$ restricted to $e_i$ is the smooth, regular parametrization of  
 a curve in $\R^2$ up to the endpoints, with self-intersections possibly only at the
 endpoints, or
\item[ii)] $\Psi$ is degenerate, i.e. it maps to a fixed point $(p, v) \in T_pR^2$ with
  $|v|=1$, for some $p \in \R^2$.
\end{itemize}
If the edge is homeomorphic to a half-line we assume the first
case. At each vertex we assume that the three tangent directions of
the curves meeting there form $120$-degree angles.  We call $(G,\Psi)$ as
above a {\it degenerate regular network} if there exists a sequence of
homeomorphisms $\Psi_k:G\rightarrow \R^2$ as above, such that $\Psi_k \rightarrow \Psi$ in $C^1$,
where we assume that the $\Psi_k$ are actually  embeddings,
i.e. $\Psi_k(G)$ are regular networks. If one or several edges
are mapped under $\Psi$ to a single point $p$ in $\R^2$, we call this
sub-network the core at $p$. Note that for a degenerate regular
network, the core at a point $p$ is always a connected sub-network.
\end{defi}

We call $(G,\Psi)$ a
generalized self-similarly shrinking network, if $\Psi$ is a degenerate
embedding in the sense above, and $\Psi|_{e_i}$ satisfies the
self-shrinker equation
$$\vec{k} = - \frac{x^\perp}{2} ,$$
for all $1\leq i \leq N_1$. The evolving self-similar solution for
$t\in (-\infty, 0)$ is then given by
$$N_t = \sqrt{-t}\Psi\ .$$
The most basic example of a generalized self-similarly shrinking
solution with triple points is the union of three half-lines, meeting
at the origin under a 120-degree condition. We will call this solution
the {\it standard triod}.

The following Lemma is from \cite{HaettenschweilerDiplom07}, for the
convenience of the reader we give the proof in here.

\begin{lemm}[H\"attenschweiler] Let $(G,\Psi)$ be a generalized self-similarly shrinking network,
such that $G$ is a tree. Then $\Psi(G)$ consists of half-lines
emanating from the origin, with possibly a core at the origin.
\end{lemm}
\begin{proof} First note that any non-degenerate self-similar
  shrinking curve is a member of the one-parameter
 family of curves classified by Abresch and Langer in
  \cite{AbreschLanger}. Their classification result implies the
  following. If the curve contains the origin, then it is a
  straight line through the origin. Otherwise 
  it is contained in a compact subset of $\R^2$, but it is still
  diffeomorphic to a line. In the latter case, any
  such curve has a constant winding direction with respect to the
  origin. Aside from the circle, any other solution has a countable,
  non-vanishing  number of self-intersections.

  Let us consider $\Gamma^\prime
  \subset \Gamma := \psi(G)$, which consists of $\Gamma$ with all
  half-lines going to infinity removed. For $\theta \in \mathbb{S}^1$ let
  $S(\theta)$ be the half-line, emanating from the origin in direction of
  $\theta$. Consider
  $$R(\theta):=\sup\{|x|\, | x \in \Gamma^\prime\cap S(\theta)\}\ .$$
  If $\Gamma^\prime$ is not only the core at the origin, there exists a
  $\phi_0$ such that $R(\phi_0) = |S(\phi)\cap \gamma_i| >0$, where
$\gamma_i$ is a non-degenerate curve of $\Psi$. Since the $\gamma_i$'s don't
change their winding direction we have
$$R(\phi)=|\gamma_i(\phi)|$$ for all $\phi \in \{\phi\ | \
\gamma_i\cap S(\phi) \neq \emptyset\}$. At an endpoint of $\gamma_i$ we
have $R(\cdot)>0$ otherwise $\gamma$ would have been a half-line,
starting at the origin. At this endpoint, also if it has a core, there is
always another $\gamma_{i^\prime}$ which continues smoothly with the
same winding direction, $R(\phi)$ stays positive and
$$R(\phi)=|\gamma_{i^\prime}(\phi)|$$ for all $\phi \in \{\phi\ | \
\gamma_{i^\prime}\cap S(\phi) \neq \emptyset\}$. This also implies
that $R(\phi)$ is continuous. Continuing until $\phi$ reaches again
$\phi_0$ we find a closed non-contractible loop in $\Gamma^\prime$,
which yields a contradiction.
\end{proof}

Let us assume that $(\gamma_t)_{0\leq t < T}$ is a network
flow. Huisken's monotonicity formula implies that the function
$$\Theta_{x_0,t_0}(t):=\int_{\gamma_t}\rho_{x_0,t_0}\, ds$$
is decreasing in time for $t<t_0$, and the limit
$\Theta(x_0,t_0):=\lim_{t\nearrow t_0}\Theta_{x_0,t_0}(t)$ is the
Gaussian density at $(x_0,t_0)$. The function
$\Theta_{x_0,t_0}(t)$ is constant in time if and only if the evolving network
is a self-similarly shrinking network, centered at the space-time
point $(x_0,t_0)$. 

The Gaussian density of the shrinking sphere can easily computed to be
$$\Theta_{S^1}=\sqrt{\frac{2\pi}{e}}\, .$$ 
Note that $\Theta_{S^1}>3/2$. For a generalized self-similar shrinking
network $\gamma$ we denote $\Theta_\gamma:=\int_\gamma
\rho_{0,0}(\cdot,-1)\, ds$.

\begin{lemm}
  \label{thm:densitybound} Let $\gamma$ be a generalized
  self-similarly shrinking network and assume that
  $\Theta_{\gamma}<\Theta_{S^1}$. Then $\gamma$ is tree-like, and thus
  either a multiplicity one line, or the standard triod. 
\end{lemm}
\begin{proof}
  By the work of Colding-Minicozzi, \cite{ColdingMinicozzi12}, it holds that
\begin{equation}\label{eq:locreg.0.0}
\Theta_{\gamma} = \int_\gamma \rho_{0,0}(\cdot,-1)\, ds = \sup_{x_0\in\R^2,
  t_0>-1} \int_\gamma \rho_{x_0,t_0}(\cdot,-1)\, ds\, .
\end{equation}
Assume that $\gamma$ is not tree-like. Let us first assume that the
complement of $\gamma$ in $\R^2$ contains no bounded component. It is easy
to see from the proof of the previous lemma that this implies that
$\gamma$ consists of at least six half-lines emanating from the origin, together
with a core. Thus would imply that $\Theta_{\gamma}\geq 3$, a
contradiction.

Let $B$ be a bounded component of the complement of $\gamma$ and
$\tilde{\gamma}$ the sub-network of $\gamma$ which bounds $B$, counted
with unit multiplicity. Since $\tilde{\gamma}$ is smooth with corners, 
and no triple
junctions, we can evolve it by classical curve shortening flow until it
shrinks at $(x_0,t_0)$ to a 'round' point. By the monotonicity formula
this implies that
$$ \int_{\tilde{\gamma}}\rho_{x_0,t_0}\, ds \geq \Theta_{S^1}\, .$$
By \eqref{eq:locreg.0.0} this implies
$$\Theta_{\gamma}\geq \int_\gamma  \rho_{x_0,t_0}\, ds \geq
\int_{\tilde{\gamma}}\rho_{x_0,t_0}\, ds \geq \Theta_{S^1}\, .$$
\end{proof}

Given a sequence $\lambda_i\nearrow \infty$ and a space-time point
$(x_0,t_0)$, where $0<t_0\leq T$ the standard parabolic
rescaling around $(x_0, t_0)$ of the
flow is given by
$$\gamma^i_\tau = \lambda_i\big(\gamma_{\lambda^{-2}_i\tau +t_0}-x_0\big)\ ,$$
where $\tau \in [-\lambda_i^{2} t_0,\, \lambda_i^{2}(T-t_0))$.
Recall that the monotonicity formula implies
$$\Theta_{x_0,t_0}(t)-\Theta(x_0,t_0)=\int\limits_t^{t_0}\int\limits_{\gamma_{\sigma}}
\Big|\vec{k}+ \frac{(x-x_0)^\perp}{2(t_0-\sigma)}\Big|^2\rho_{x_0,t_0}(\cdot,\sigma)\, ds\,d\sigma$$
Changing variables according to the parabolic rescaling, we obtain
$$\Theta_{x_0,t_0}(t)-\Theta(x_0,t_0)=\int\limits_{\lambda_i^2(t-t_0)}^{0}\int
\limits_{\gamma^i_{\tau}}
\Big|\vec{k}- \frac{x^\perp}{2\tau}\Big|^2\rho_{0,0}(\cdot,\tau)\,
ds\,d\tau\ , $$
or for fixed time $\tau_0 \in (-\lambda_i^{2}t_0,0)$,
\begin{equation}\label{monresc}\Theta_{x_0,t_0}(t_0+\lambda_i^{-2}\tau_0)-\Theta(x_0,t_0)=
\int\limits_{\tau_0}^{0}\int
\limits_{\gamma^i_{\tau}}
\Big|\vec{k}- \frac{x^\perp}{2\tau}\Big|^2\rho_{0,0}(\cdot,\tau)\,
ds\,d\tau\ . 
\end{equation}
We now give a slightly modified version of the Blowup-Lemma in 
\cite{HaettenschweilerDiplom07}.

\begin{lemm} \label{thm:shrinkingnetworks.1} There exists a subsequence $(\lambda_{i})$ (relabeled again
the same) such that for almost all $\tau \in (-\infty, 0)$ and for any
$\alpha \in (0,1/2)$ 
$$\gamma^i_\tau \rightarrow \bar\gamma_\tau$$
in $C^{1,\alpha}_\text{loc}\cap W^{2,2}_\text{loc}$, where $\bar{\Gamma}_\tau$ is a generalized
self-similarly shrinking network at time $\tau$. This convergence
also holds in the sense of radon measures for all $\tau$. Note that the subsequence
does not depend on $\tau$ and also not the limit (except for scaling). 
\end{lemm}
\begin{proof} We first choose a subsequence such that the rescalings converge 
as Brakke flows to a self-similarly shrinking tangent flow. Let 
$$f_i(\tau):=\int\limits_{\gamma^i_{\tau}}
\Big|\vec{k}- \frac{x^\perp}{2\tau}\Big|^2\rho_{0,0}(\cdot,\tau)\,
ds .$$
Note that \eqref{monresc} implies that $f_i\rightarrow 0$ in
$L^1_\text{loc}((-\infty, 0])$. Thus there exists a subsequence 
such that $f_i$ converges point-wise a.e. to zero. This implies that
that for any $R>0$
$$\int\limits_{\gamma^i_{\tau}\cap B_R(0)}
|k|^2ds \leq C\, ,$$
independent of $i$. By choosing a further subsequence
 we can assume that $\Gamma^i_\tau$ converges in
 $C^{1,\alpha}_\text{loc}$ to a degenerate network. Note that each
 limiting segment, which is non-degenerate is in $W^{2,2}_\text{loc}$
 and and is a weak solution of 
$$\vec{k} =  \frac{x^\perp}{2\tau}\, .$$
By elliptic regularity, each such segment is actually smooth, and thus
the limiting network is a generalized self-similarly
shrinking network at time $\tau$. But since this limit has to coincide in measure
with the limiting Brakke-flow it is unique, and the whole sequence
converges. The convergence in $W^{2,2}_\text{loc}$ is implied by the weak
convergence in $W^{2,2}$ and the fact
that $f_i(\tau) \rightarrow 0$. 
\end{proof}

This can be strengthened, if the limit has unit density.

\begin{lemm}\label{thm:shrinkingnetworks.2} Assume that a sequence of rescalings as above converges in the sense of
Brakke flows to a regular self-similarly shrinking network,
$$ (\gamma^i_t) \rightarrow (\bar{\gamma}_t)\ .$$ 
 Then this convergence is smooth on all compact subsets of
$\R^2\times (-\infty, 0)$.
\end{lemm}
\begin{proof}
By the Lemma before we can choose a further subsequence such that we
have $\Gamma^i_\tau \rightarrow \bar{\Gamma}_\tau$ in
$C^{1,\alpha}_\text{loc}\cap W^{2,2}_\text{loc}$ for almost every
$\tau$. Now take any set of the form $\Omega = \bar{B}_R(0)\times
[a,b]$, $a<b<0$, where we choose $R$ big enough, such that
$\partial B_R(0)$ intersects $\bar{\Gamma}_\tau$ for $\tau \in [a-2,b]$
only in the straight lines going out to infinity (if they
exist). Since for almost every $\tau$ we have convergence in
$C^{1,\alpha}$ we know that the Gaussian density ratios in this set are
less than $1+\varepsilon$ for all $\tau \in [a-3/2,b]$. Thus there can
be no triple points present, and by the estimates of White
\cite{White05}, we can choose $i_0$ big enough such that
 $|\nabla^j k|$ is small on $\Gamma_i$ for all $j\geq 0$ on
$B_{R+1}\setminus B_R$ for all $\tau \in [a-1,b]$ and all
$i>i_0$.
 Now
for any given $\epsilon,\delta > 0$ we choose $i_0$ even bigger such that there
exists times $\tau_j,\ j = 0,\dots, N:=2 [(b-a+1)/\delta]+1$ such that
$$ |\tau_0-a+1|, |\tau_N-b| \leq \delta,\ |\tau_{j+1}-\tau_j| \leq \delta$$   
for all $0\leq j \leq N-1$ and
$$\|\Gamma^i_{\tau_j}-\bar{\Gamma}_{\tau_j}\|_{W^{2,2}(B_{R+1})} \leq
  \epsilon \ ,$$
for all $i>i_0$. We now fix $\epsilon >0$ and adjust $\delta>0$
accordingly
such that we can ensure that by \eqref{odek2linv} that  
$$\|\Gamma^i_{\tau}\|_{W^{2,2}(B_{R+1})} \leq
  C \ ,$$
for all $\tau \in [a,b]$ and $i>i_0$. The higher order interior estimates then
prove smooth subsequential convergence. For this argument we had chosen a subsequence, but since we can
always choose such a subsequence, the whole sequence converges.
 \end{proof}

\subsection{Regularity results}\label{sec:reg res}
\medskip
In the following we will give some regularity results for 'proper'
flows, where we say that a flow is a proper flow, given by its space-time
track $\mathcal{M}$ in an open subset $U$ of
space-time, if 
$$\mathcal{M} = \overline{\mathcal{M}}\cap U\ ,$$
compare with section 2.3 in \cite{White05}.
\begin{thm}\label{thm:locreg.3}
Let $(\gamma^i_t)_{0<t<T}$ be a sequence of smooth network flows with
uniformly bounded length ratios
which converges locally as Brakke flows to the standard triod. Then this
convergence is smooth.
\end{thm}
\begin{proof}
We can assume that the triple point of the standard triod is at the
origin. Fix $T^\prime>T$. Then for any $0<t_1<t_2<T$ we have
$$\int_{\gamma^i_{t_j}} \Phi_{0,T^\prime}\, ds \rightarrow \frac{3}{2}
  \qquad \text{where}\ j=1,2\ .$$
Then as in the proof of lemma \ref{thm:shrinkingnetworks.1} we obtain
that there is subsequence where we have $C^{1,\alpha}$-convergence to
the standard triod for a.e. $t\in (t_1,t_2)$. Note that by White's
regularity theorem the convergence is smooth away from the triple
point, and that out of combinatorial reasons no core can develop at
the triple point of the standard triod. As in the proof of lemma
\ref{thm:shrinkingnetworks.2} we can then show that the
  convergence is actually smooth.
\end{proof} 
In the following we will prove a local regularity result in the spirit of
Brian White's result for mean curvature flow \cite{White05}. We follow
here the alternative proof of Ecker \cite[Theorem 5.6]{Ecker04}. The
Gaussian density ratios are defined as
\begin{equation}\label{Gaussian density}
 \Theta(x,t,r):=\Theta_{t-r^2}(x,t)\ .
\end{equation}
In the case of proper flows, which are only defined in an open subset
of space-time one has to localize Huisken's monotonicity
formula. Compare with section 10 in \cite{White97} and Remark 4.16
together with Proposition 4.17
in \cite{Ecker04}. To keep the presentation simpler, we will only
give proofs of the next two theorems for flows defined on all of
$\R^2$ and leave the modifications in the case of proper flows to the reader.

\begin{thm}\label{thm:locreg.1} Let $(\gamma_t)_{t\in[0,T)}$ be a
  smooth, proper and regular planar network flow in $ B_\rho(x_0)\times (t_0-\rho^2,t_0)$ which reaches the point $x_0$ at time
  $t_0 \in (0,T]$. Assume that for some $\varepsilon
  >0$  it holds that 
\begin{equation}\label{eq:locreg.0}
\Theta(x,t,r) \leq 2- \varepsilon
\end{equation}
for all $(x,t) \in B_\rho(x_0)\times (t_0-\rho^2,t_0)$ and
$0<r<\eta\rho$ for some $\eta >0$, where $(1+\eta)\rho^2\leq
t_0<T$. Furthermore, assume that $\gamma_t\cap B_\rho(x_0)$ has no closed
loops of length less than $\delta \rho >0$
for all $t\in (t_0 - (1+\eta)\rho^2,t_0)$ for some $\delta>0$. Then
there exists $C=C(\epsilon, \eta, \delta)$ such that
$$ |k|^2(x,t) \leq \frac{C}{\sigma^2\rho^2} $$
for $ (x,t)\in \big(\gamma_t\cap
    B_{(1-\sigma)\rho}(x_0)\big)\times (t_0-(1-\sigma)^2\rho^2,t_0)$
    and all $\sigma \in (0,1)$.
\end{thm}
\begin{rmk}\label{rem:locreg.1}
  Note that the bound on the curvature, together with the balancing
  condition and \eqref{eq:locreg.0}, gives that there is a constant
  $\kappa=\kappa(\varepsilon, \eta, \delta)>0$ such that
  the length of each segment which intersects
  $B_{(1-\sigma)\rho}(x_0)\times (t_0-(1-\sigma)^2\rho^2,t_0)$ is
  bounded from below by $\kappa\cdot \sigma \rho$.
This implies, using Theorem \ref{thm:locreg}, corresponding scaling invariant estimates on all higher derivatives of the curvature.
\end{rmk}
\begin{proof} We can first assume that $t_0<T$, and pass to limits
  later. By translation and scaling we can furthermore assume that $x_0=0$
  and $\rho = 1$. We can now follow more or less verbatim the proof of
  Theorem 5.6 in \cite{Ecker04}. Supposing that the statement is not
  correct we can find a sequence of smooth, regular network flows
  $(\gamma^j_t)$, defined for $t \in [-1-\eta,0]$, reaching the point
  $(0,0)$ and satisfying the
  above conditions, but 
  \begin{equation}
    \label{eq:locreg.0.1}
    \zeta_j^2 := \sup_{\sigma \in (0,1)}\Bigg(\sigma^2
    \sup_{(-(1-\sigma)^2,0)}\sup_{\gamma^i_t\cap B_{1-\sigma}} |k|^2
      \Bigg) \rightarrow \infty
  \end{equation}
as $j \rightarrow \infty$. We can find $\sigma_j \in (0,1)$ such that
$$\zeta_j^2 = \sigma_j^2
    \sup_{(-(1-\sigma_j)^2,0)}\sup_{\gamma^i_t\cap B_{1-\sigma_j}} |k|^2$$
and $y_j \in \gamma^j_{\tau_j} \cap \bar{B}_{1-\sigma_j}$ at a time
$\tau_j \in [-(1-\sigma_j)^2,0]$ so that
 \begin{equation}
    \label{eq:locreg.0.2}
\zeta_j^2 = \sigma_j^2 |k(y_j,\tau_j)|^2\ .
\end{equation}
We now take 
$$\lambda_j = |k(y_j,\tau_j)|^{-1}$$
and define
$$\tilde{\gamma}^j_s = \frac{1}{\lambda_j}
\Big(\gamma^j_{\lambda_j^2s+\tau_j} - y_j\Big)$$
for $s\in [-\lambda_j^2\sigma_j^2/4,0]$. As in the the proof of
Theorem 5.6 in \cite{Ecker04} we see that 
\begin{equation}\label{eq:locreg.0.2a}
0\in \tilde{\gamma}^j_0\, , \qquad |k(0,0)|=1
\end{equation}
and
$$\sup_{(-\lambda_j^{-2}\sigma_j^2/4,0)}\sup_{\tilde{\gamma}^j_s
  \cap B_{\lambda_j^{-1}\sigma_j/2}} |k|^2 \leq 4$$
for every $j \geq 1$. Since $\lambda_j^{-2}\sigma_j^2 = \zeta_j^2
\rightarrow \infty$ we see that up to a subsequence, labeled
again the same, 
\begin{equation}\label{eq:locreg.0.3}
\tilde{\gamma}_s^j \rightarrow \tilde{\gamma}^\infty_s 
\end{equation} 
converges locally uniformly
$\mathbb{R}^2\times\mathbb{R}$ in $C^{1,\alpha}$ to a limiting
$C^{1,1}$-solution $\tilde{\gamma}^\infty_s$ of the network flow. Note
that this limiting solution is defined for $s \in (-\infty,0]$ and
possibly degenerate, i.e. cores and higher density lines can develop. But note
that \eqref{eq:locreg.0} implies that on $\tilde{\gamma}^\infty_s$
\begin{equation}\label{eq:locreg.0.4}
 \Theta(x,t,r)\leq 2-\epsilon
\end{equation}
for all $r>0$ and $(x,t)$ in $\R^2\times (-\infty,0]$.
Together with the fact that $\tilde{\gamma}^\infty_s$ is uniformly
bounded in $C^{1,1}$, this implies that $\tilde{\gamma}^\infty$ is non-degenerate, i.e. there are no higher densities and no cores. Furthermore,
the assumption on the lower bound for the length of closed loops
implies that $\tilde{\gamma}^\infty$ is tree-like. 
The estimate \eqref{eq:locreg.0.4} yields that the tangent flow at
$-\infty$ is either a static unit density line,
or the standard triod and $\lim_{r\rightarrow
  -\infty}\Theta(x,t,r)$ is either $1$ or $3/2$. In the first case
this implies that $\tilde{\gamma}^\infty$ is a static unit density
 line. But then White's local regularity
theorem implies that the convergence in \eqref{eq:locreg.0.3} is
smooth. This gives a contradiction to  \eqref{eq:locreg.0.2a}. In the
second case $\tilde{\gamma}^\infty$ has to have a triple point, and
thus is the standard triod. Then Proposition \ref{thm:locreg.3} gives a
contradiction as before.
\end{proof}

Without the assumption on the length of the shortest loops, we prove a similar
statement if the Gaussian densities are less than $\Theta_{S^1}$:
\begin{proof}[Proof of Theorem \ref{thm:locreg.2}] 
The proof is nearly identical to the proof of Theorem
\ref{thm:locreg.1}. Rescaling and translating such that $x_0=0$ and $\rho = 1$ we
assume that as a contradiction
 we have a sequence of smooth, regular network flows
$(\gamma^j_t)$ defined for $t\in [-1-\eta,0]$, reaching $(0,0)$,
satisfying \eqref{eq:locreg.0.1} and
\begin{equation}
  \label{eq:locreg.0.6}
  \Theta(x,t,r)\leq \Theta_{S^1} -\varepsilon
\end{equation}
for all $(x,t) \in B_1 \times (-1,0)$ and $0<r<\eta$. Rescaling as before we obtain a
limiting $C^{1,1}$ solution $(\tilde{\gamma}^\infty_s)$ which
satisfies
$$\Theta(x,t,r)\leq \Theta_{S^1} -\varepsilon$$
for all $r>0$ and $(x,t) \in \mathbb{R}^2(-\infty,0]$. By Lemma
\ref{thm:densitybound} this implies that the tangent flow at $-\infty$
is either a static unit density line, or the standard triod. We reach
a contradiction as in the proof of the previous theorem.
\end{proof}

\section{A pseudolocality result for Mean Curvature Flow}\label{sec:graph_local}
We recall the following setup from the introduction. For any point $x \in \R^{n+k}$ we write $x = (\hat{x},
\tilde{x})$ where $\hat{x}$ is the orthogonal projection of $x$ on the
$\R^n$-factor and $\tilde{x}$ the orthogonal projection on the $\R^k$
factor.  We define the cylinder $C_R(x_0) \subset \R^{n+k}$ by
$$C_r(x) = \{x\in \R^{n+k}\, |\, |\hat{x}-\hat{x}_0|<r,
|\tilde{x}-\tilde{x}_0|<r\}\ .$$
Furthermore, we write $B_r^n(x_0) = \{(\hat{x},\tilde{x}_0)\in
\R^{n+k}\, \ \, |\hat{x}-\hat{x}_0|<r\}$. 

\begin{proof}[Proof of Theorem \ref{thm:graph_local}] We first assume
  that $T\geq 1$. Translating
  $x_0$ to $0$ and rescaling with a factor $R>0$ we can assume that
  $M_0\cap C_R(0) = \text{graph}(u)$ where $u:B^n_R(0)\rightarrow
  \R^k$ with Lipschitz constant less than $\varepsilon$. We want to
  show that there exists $R\gg 1$ such that $M_t\cap C_1(0)$
is a graph with Lipschitz constant less than $\eta$ and height bounded
by $\eta/2$ for all $t\in[0,1]$. 

Recall that by \cite{White05}, if all Gaussian density ratios up
to scale $1$ centered at $(x,t) \in B_2(0)\times [0,1]$ are bounded
above by $1+\varepsilon_0$ then 
$$|A|_{M_1}(x) \leq C(\varepsilon_0)$$
for all $x\in M_1\cap C_1(0)$. Furthermore, a compactness argument
implies that $C(\varepsilon_0) \rightarrow 0$ as $\varepsilon
\rightarrow 0$. This implies that we can choose
$0<\varepsilon_1<\varepsilon_0$ such that if the Gaussian density
ratios up to scale $1$ centered at $(x,t) \in B_2(0)\times [0,1]$ are bounded by $1+\varepsilon_1$  and 
$$M_1\cap C_1(0) \subset C_1\cap \{|\tilde{y}|\leq \eta/2\}\ , $$
then $M_1 \cap C_1(0)$ is a graph over $B^n_1(0)$ with Lipschitz
constant bounded above by $\eta$. 

Now assume that $y\in C_2(0)\cap \{|\hat{y}|\leq \eta/2\}$. We then
have for $R\geq 2, r\leq 1$ that
\begin{equation}\label{eq:gl.1}
  \begin{split}
    \Theta_0(y,r) &= \int\limits_{M_0\cap C_R}\frac{1}{(4\pi
      r^2)^{n/2}}e^{-\frac{|x-y|^2}{4r^2}}\, d\H^n(x)\\
&\ \ \  +  \int\limits_{M_0\setminus C_R}\frac{1}{(4\pi
      r^2)^{n/2}}e^{-\frac{|x-y|^2}{4r^2}}\, d\H^n(x)\\
&\leq \int\limits_{B^n_R(0)}\frac{1}{(4\pi
      r^2)^{n/2}}e^{-\frac{|\hat{x}-\hat{y}|^2}{4r^2}}\Big(\det(\, \pmb{1}+Du^\top\circ
    Du)\Big)^{1/2}\, d\hat{x}^n \\ 
&\ \ \ + e^{-\frac{(R-3)^2}{8r^2}} \int\limits_{M_0\setminus C_R}\frac{1}{(4\pi
      r^2)^{n/2}}e^{-\frac{|x-y|^2}{8r^2}}\, d\H^n(x)\\
&\leq (1+\varepsilon^2)^{n/2} + C e^{-\frac{(R-3)^2}{8r^2}} \leq 1 +
  \varepsilon_1\ ,
  \end{split}
\end{equation}
provided $\varepsilon\leq \varepsilon_2$ and $R\geq R_0\geq 3$. 

By assumption we have that $\sup_{B^n_R(0)}|u|\leq \varepsilon R$. Let
us assume that $\varepsilon$ is small enough, depending on $R$, such that
\begin{equation}
  \label{eq:gl.2}
  \varepsilon R \leq \eta/4\ . 
\end{equation}
Now let $y \in C_2(0)\setminus \{|\tilde{y}|\leq \eta/2\}$. We can then
estimate 
\begin{equation*}
  \begin{split}
    \Theta_0(y,r) &= \int\limits_{M_0\cap C_R}\frac{1}{(4\pi
      r^2)^{n/2}}e^{-\frac{|x-y|^2}{4r^2}}\, d\H^n(x)\\
&\ \ \  +  \int\limits_{M_0\setminus C_R}\frac{1}{(4\pi
      r^2)^{n/2}}e^{-\frac{|x-y|^2}{4r^2}}\, d\H^n(x)\\
&\leq \int\limits_{B^n_R(0)}\!\!\!\frac{1}{(4\pi
      r^2)^{n/2}}e^{-\frac{|u(\hat{x})-\tilde{y}|^2+ |\hat{x}-\hat{y}|^2}{4r^2}}\Big(\det(\, \pmb{1}+Du^\top\circ
    Du)\Big)^{1/2} d\hat{x}^n\\
&\ \ \ + e^{-\frac{(R-3)^2}{8r^2}} \int\limits_{M_0\setminus C_R}\frac{1}{(4\pi
      r^2)^{n/2}}e^{-\frac{|x-y|^2}{8r^2}}\, d\H^n(x)\\
&\leq e^{-\frac{\eta^2}{64 r^2}} (1+\varepsilon^2)^{n/2} + C
  e^{-\frac{(R-3)^2}{8r^2}}
  \end{split}
\end{equation*}
Now we choose $R=R_1\geq R_0$ such that 
$$C e^{-\frac{(R_1-2)^2}{8}} < 1- e^{-\frac{\eta^2}{128}}\ .$$
Then choose $\varepsilon \leq \min\{\varepsilon_2, 4^{-1} \eta
R_1^{-1}\}$ such that
$$e^{-\frac{\eta^2}{64}}(1+\varepsilon^2)^{n/2} <
e^{-\frac{\eta^2}{128}} \ .$$
Note that the first assumption on $\varepsilon$ implies that
\eqref{eq:gl.2} is satisfied. We see that
that $\Theta_0(y,r) <1$ for all $y\in C_2\setminus \{|\hat{y}|\leq
\eta/2\}$. Using the monotonicity of the Gaussian density ratios this
yields that
$$M_t\cap C_2(0) \subset C_2(0)\cap \{|\hat{y}|\leq
\eta/2\}\ .$$
for all $t\in [0,1]$ and by the choice of $\varepsilon$ and estimate \eqref{eq:gl.1} that 
$$\Theta(x,t,r) \leq 1 +\varepsilon_1$$
for all $(x,t) \in C_2(0)\times [0,1]$ and scales $r$ up to one. By
the choice of $\varepsilon_1$ this gives that $M_1\cap C_1(0)$ is a smooth
graph over $B^n_1(0)$ with Lipschitz constant bounded above by
$\eta$. 

We want to show that this is also true for $M_t\cap C_1(0)$ for any
$0<t<1$. Pick $t_0 \in (0,1)$ and let $\lambda = t_0^{-1/2}$. Let
$(M^\lambda_t)_{0\leq t \leq \lambda^2T}$ be the flow, parabolically
rescaled by $\lambda$. Note
that for any $x_0\in M^\lambda_0\cap C_{(\lambda-1) R_1}(0)$ we can
shift $x_0$ to $0$ and see that our previous assumptions are satisfied
for this flow. That yields that $M^\lambda_1\cap C_1(x_0)$ is a smooth
graph over $B^n_1(\hat{x}_0)$ with Lipschitz constant bounded above by
$\eta$. Note that this property is scaling invariant. Scaling back
this implies that $M_{t_0}\cap C_{(1-t_0^{1/2})R_1+t_0^{1/2}}(0)$ is a
graph over $B^n_{(1-t_0^{1/2})R_1+t_0^{1/2}}(0)$ with Lipschitz constant
less than $\eta$. Since $(1-t_0^{1/2})R_1+t^{1/2}\geq 2 - t_0^{1/2} \geq
1$ this implies the statement. 

If $T<1$ we can first rescale the flow by a factor $\lambda =
T^{-1/2}$ as above and then scale back to get the result for $0<t<T$.
\end{proof}

\section{Appendix}
We derive some useful technical results.  In what follows $\sigma_0$ is
a regular network with $(\sigma_t)_{0\leq t<T}$  being a  regular solution for network
flow. Moreover,  $\chi$ is a fixed regular network in  $\R^2$ and $\varepsilon_0$ is a universal constant less than $1/2$.
\begin{lemm}\label{shortime}
  Fix $\varepsilon_0$ and $\alpha$. There exist
  $\varepsilon=\varepsilon(\chi,\varepsilon_0, \alpha)$ and
  $q_1=q_1(\chi, \varepsilon_0, \alpha)$ so that, for every $R\geq 2$, if $\sigma_0$ is
  $\varepsilon$-close to $\chi$ in $C^{1,\alpha}(B_{R}(0)),$ then for
  every $r^2,t\leq q_1$ and $y\in B_{R-1}(0)$
  $$\Theta_{t}(y,r)\leq 3/2+\varepsilon_0.$$
\end{lemm}
\begin{proof}
We argue by contradiction. Suppose there are  sequences $(\varepsilon_i)_{i\in\N}$, $(r_i)_{i\in \N}$, $(t_i)_{i\in\N}$ all converging to zero, $(R_i)_{i\in N}$ with $R_i\geq 2$ for all $i$, $(y_i)_{i\in\N}$ with $y_i \in B_{R_i-1}(0)$,  and  $(\sigma_t^i)_{0\leq t \leq t_i}$ a sequence of regular solutions to  network flow for which $\sigma_0^i$  is $\varepsilon_i$-close to $\chi$ in $C^{1,\alpha}(B_{R_i}(0))$ and
$$\Theta_{t_i}(y_i,r_i)> 3/2+\varepsilon_0.$$
The fact that $\sigma_i$ is $\varepsilon_i$-close to $\chi$ in $C^{1,\alpha}(B_{R_i}(0))$ means that
\begin{itemize}
\item[a)] there are functions $u_i$ defined on $C^0(\chi\cap B_{R_i}(0))$ which are $C^{1,\alpha}$ when restricted to each branch and the $C^{1,\alpha}$ norm on each branch converges to zero;
\item[b)] there are unit vectors $N_i$ defined on $\chi\cap B_{R_i}(0)$  such that $\langle N_i,\nu\rangle$ is a smooth function on each branch that converges uniformly to one, where $\nu$ denotes a unit normal vector on the respective branch;
\item[c)]  
$$\sigma_0^i\cap B_{R_i}(0)=\chi\cap B_{R_i}(0)+u_i N_i.$$
\end{itemize}
Set $\lambda_i=\sqrt{r_i^2+t_i}$. It is simple to recognize that we can find functions $v_i$ defined on $\lambda_i^{-1}\chi$ such that 
$$\lambda_i^{-1}\sigma_0^i\cap B_{R_i \lambda_i^{-1}}(0)=\lambda_i^{-1}\chi\cap B_{R_i \lambda_i^{-1}}(0)+v_iN_i.$$
Because the $C^{0,\alpha}$ norm of the first derivatives of $v_i$ converges uniformly to zero on each branch of $\lambda_i^{-1}\chi\cap B_{R_i \lambda_i^{-1}}(0)$ we obtain that if the limit of $\lambda_i^{-1}(\sigma_0^i-y_i)$ in the varifold sense is not empty, then it must be either a line or three half-lines meeting at a common point. In any case we have
$$\lim_i \int_{\lambda_i^{-1}(\sigma_0^i-y_i)} (4\pi)^{-1/2}\exp(-|x|^2/4)d\H^1\leq 3/2.$$
This contradicts the fact that, using the monotonicity formula, 
\begin{equation*}\begin{split}
 3/2+\varepsilon_0 &<\Theta_{t_i}(y_i,r_i)\leq \Theta_{0}(y_i,\lambda_i)\\
 &=\int_{\lambda_i^{-1}(\sigma_0^i-y_i)} (4\pi)^{-1/2}\exp(-|x|^2/4)d\H^1.
 \end{split}\end{equation*}
\end{proof}
Assume that $\sigma$ is non-compact, asymptotic to half-lines at infinity, and  contains no closed loops.
\begin{lemm}\label{junction}
  Fix $\varepsilon_0$, $R>2$, and $\tau$. There is $\eta=\eta(\tau)$ such that if
  $$\Theta(x,r)\leq 3/2+\varepsilon_0<2$$
  for every $x$ in $B_{R}(0)$ and $r^2\leq \tau$, then the distance
  between any two triple junctions in $B_{R}(0)$ is greater than $\eta$.
\end{lemm}
\begin{proof}
 Choose $T>0$ so that
 $$\int_{0}^T  (4\pi)^{-1/2}\exp(-t^2/4)dt=1/2+\varepsilon_0/16-1/32 $$
 and $r_1>0$ so that
 $$\int_{0}^{r_1}  (4\pi)^{-1/2}\exp(-t^2/4)dt=1/32-\varepsilon_0/16. $$
Suppose that $\sigma$ has two triple junctions in $B_{R}(0)$ at a distance $\eta$ smaller than $\tau r_1$. Denote the midpoint between the triple junctions by $y$ and consider the network $\gamma=\tau^{-1}(\sigma-y)$ which has two triple junctions $x_1$ and $-x_1$ inside $B_{r_1}(0)$. Because $\gamma$ has no closed loops we can find  paths $a_1$, $b_1$, $a_2$, and $b_2$ contained in $\gamma$  such that $a_1, a_2$ and $b_1, b_2$ connect $x_1$  and $-x_1$ respectively to a point at a distance $T$ from the origin and  $a_1\cap a_2=\{x_1\}$, $b_1\cap b_2=\{-x_1\}$, $a_i\cap b_j=\emptyset$.  

Consider the metric $g=(4\pi)^{-1}\exp(-|x|^2/2)(dx_1^2+dx_2^2)$ on $\R^2$ and denote its distance function by $d_g$. Straight lines containing the origin are geodesics for $g$ and so  any point with $|p|=T$ has 
 $$d_g(x,p)\geq 1/2+\varepsilon_0/8-1/16\quad\mbox{for any } x\in B_{r_1}(0).$$
Thus
\begin{multline*}
 3/2+\varepsilon_0\geq \Theta(y,\tau)=\int_{\gamma}(4\pi)^{-1/2}\exp(-|x|^2/4)d\H^1\\
 \geq \sum_{i=1}^2\left(\int_{a_i}(4\pi)^{-1/2}\exp(-|x|^2/4)d\H^1+\int_{b_i}(4\pi)^{-1/2}\exp(-|x|^2/4)d\H^1\right)\\
 =\sum_{i=1}^2 \left(\int_{a_i} dl_g+\int_{b_i}dl_g\right)\geq 4(1/2+\varepsilon_0/8-1/16)=7/4+\varepsilon_0/2.
\end{multline*}
This is impossible because  $\varepsilon_0< 1/2.$
 \end{proof}

\providecommand{\bysame}{\leavevmode\hbox to3em{\hrulefill}\thinspace}
\providecommand{\MR}{\relax\ifhmode\unskip\space\fi MR }
\providecommand{\MRhref}[2]{%
  \href{http://www.ams.org/mathscinet-getitem?mr=#1}{#2}
}
\providecommand{\href}[2]{#2}


\begin{thebibliography}{10}

\bibitem{AbreschLanger}
U.~Abresch and J.~Langer, \emph{The normalized curve shortening flow and
  homothetic solutions}, J. Differential Geom. \textbf{23} (1986), no.~2,
  175--196, MR0845704.

\bibitem{BellettiniNovaga11}
G.~Bellettini and M.~Novaga, \emph{Curvature evolution of nonconvex lens-shaped
  domains}, J. Reine Angew. Math. \textbf{656} (2011), 17--46, MR2818854.

\bibitem{Brakke}
K.~Brakke, \emph{The motion of a surface by its mean curvature}, Princeton
  Univ. Press, 1978, MR0485012.

\bibitem{BrendleHuisken13}
S.~Brendle and G.~Huisken, \emph{Mean curvature flow with surgery of mean
  convex surfaces in {$\mathbb{R}^3$}}, Invent. Math. \textbf{203} (2016), no.~2,
  615--654, MR3455158.

\bibitem{ChenYin07}
B.-L. Chen and L.~Yin, \emph{Uniqueness and pseudolocality theorems of the mean
  curvature flow}, Comm. Anal. Geom. \textbf{15} (2007), no.~3, 435--490, MR2379801.

\bibitem{ColdingMinicozzi12}
T.~H. Colding and W.~P. Minicozzi, II, \emph{Generic mean curvature flow {I}:
  generic singularities}, Ann. of Math. (2) \textbf{175} (2012), no.~2,
  755--833, MR2993752.

\bibitem{Ecker04}
K.~Ecker, \emph{Regularity theory for mean curvature flow}, Progress in
  Nonlinear Differential Equations and their Applications, 57, Birkh\"auser
  Boston Inc., Boston, MA, 2004, MR2024995.

\bibitem{EckerHuisken91}
K.~Ecker and G.~Huisken, \emph{Interior estimates for hypersurfaces moving by
  mean curvature}, Invent. Math. \textbf{105} (1991), no.~3, 547--569, MR1117150.

\bibitem{GageHamilton86}
M.~E. Gage and R.~S. Hamilton, \emph{The heat equation shrinking convex plane
  curves}, J. Diff. Geom. \textbf{23} (1986), 69--96, MR0840401.

\bibitem{Grayson89}
M.~A. Grayson, \emph{The heat equation shrinks embedded plane curves to points},
  J. Diff. Geom. \textbf{26} (1987), 285--314, MR0906392.

\bibitem{HaettenschweilerDiplom07}
J.~H\"attenschweiler, \emph{{Mean curvature flow of networks with triple
  junctions in the plane, {\rm Diplomarbeit, Z\"urich: ETH Z\"urich, Department
  of Mathematics, 46 p., 2007.}}}

\bibitem{Tonegawa14a}
K.~Kasai and Y.~Tonegawa, \emph{A general regularity theory for weak mean
  curvature flow}, Calc. Var. Partial Differential Equations \textbf{50}
  (2014), no.~1-2, 1--68, MR3194675.

\bibitem{MagniMantegazzaNovaga13}
A.~Magni, C.~Mantegazza, and M.~Novaga, \emph{Motion by curvature of planar
  networks, {II}}, Ann. Sc. Norm. Super. Pisa Cl. Sci. (5) \textbf{15} (2016),
  117--144, MR3495423.

\bibitem{Mantegazza04}
C.~Mantegazza, \emph{Evolution by curvature of networks of curves in the
  plane}, Variational problems in {R}iemannian geometry, Progr. Nonlinear
  Differential Equations Appl., vol.~59, Birkh\"auser, Basel, 2004, Joint
  project with Matteo Novaga and Vincenzo Maria Tortorelli, pp.~95--109.

\bibitem{MNTNetworks}
C.~Mantegazza, M.~Novaga, and V.~Maria Tortorelli, \emph{Motion by curvature of
  planar networks}, Ann. Sc. Norm. Super. Pisa Cl. Sci. (5) \textbf{3} (2004),
  no.~2, 235--324, MR2076269.

\bibitem{MazzeoSaez07}
R.~Mazzeo and M.~S{\'a}ez, \emph{Self-similar expanding solutions for the
  planar network flow}, Analytic aspects of problems in {R}iemannian geometry:
  elliptic {PDE}s, solitons and computer imaging, S\'emin. Congr., vol.~22,
  Soc. Math. France, Paris, 2011, pp.~159--173, MR3060453.

\bibitem{Neves07}
A.~Neves, \emph{Singularities of {L}agrangian mean curvature flow:
  zero-{M}aslov class case}, Invent. Math. \textbf{168} (2007), no.~3,
  449--484, MR2299559.

\bibitem{Neves13b}
\bysame, \emph{Finite time singularities for {L}agrangian mean curvature flow},
  Ann. of Math. (2) \textbf{177} (2013), no.~3, 1029--1076, MR3034293.

\bibitem{Neves13a}
A.~Neves and G.~Tian, \emph{Translating solutions to {L}agrangian mean
  curvature flow}, Trans. Amer. Math. Soc. \textbf{365} (2013), no.~11,
  5655--5680, MR3091260.

\bibitem{Saez11}
M.~S{\'a}ez~Trumper, \emph{Uniqueness of self-similar solutions to the network
  flow in a given topological class}, Comm. Partial Differential Equations
  \textbf{36} (2011), no.~2, 185--204, MR2763337.

\bibitem{Linsen}
O.~C. Schn{\"u}rer, A.~Azouani, M.~Georgi, J.~Hell, N.~Jangle, A.~Koeller,
  T.~Marxen, S.~Ritthaler, M.~S{\'a}ez, F.~Schulze, and B.~Smith,
  \emph{Evolution of convex lens-shaped networks under the curve shortening
  flow}, Trans. Amer. Math. Soc. \textbf{363} (2011), no.~5, 2265--2294, MR2763716.

\bibitem{SchnuererSchulze_2007}
O.~C. Schn{\"u}rer and F.~Schulze, \emph{Self-similarly expanding networks to
  curve shortening flow}, Ann. Sc. Norm. Super. Pisa Cl. Sci. (5) \textbf{6}
  (2007), no.~4, 511--528, MR2394409.

\bibitem{Tonegawa14b}
Y.~Tonegawa, \emph{A second derivative {H}\"older estimate for weak mean
  curvature flow}, Adv. Calc. Var. \textbf{7} (2014), no.~1, 91--138, MR3176585.

\bibitem{TonegawaWickramasekera14}
Y.~Tonegawa and N.~Wickramasekera, \emph{The blow up method for {B}rakke flows:
  networks near triple junctions}, Arch. Ration. Mech. Anal. \textbf{221}
  (2016), no.~3, 1161--1222, MR3508999.

\bibitem{White97}
B.~White, \emph{Stratification of minimal surfaces, mean curvature flows, and
  harmonic maps}, J. reine angew. Math. \textbf{488} (1997), 1--35, MR1465365.

\bibitem{White05}
\bysame, \emph{A local regularity theorem for mean curvature flow}, Ann. of
  Math. (2) \textbf{161} (2005), no.~3, 1487--1519, MR2180405.

\end{thebibliography}
\end{document}